\documentclass[11pt]{amsart}

\usepackage[a4paper,hmargin=3.5cm,vmargin=4cm]{geometry}
\usepackage{amsfonts,amssymb,amscd,amstext}
\usepackage{graphicx}
\usepackage[dvips]{epsfig}
\usepackage{quoting}
\usepackage{hyperref}

\usepackage{fancyhdr}
\input xy
\xyoption{all}

\usepackage{enumerate}
\usepackage{titlesec}
\usepackage{mathrsfs}

\titleformat{\section}
{\filcenter\bfseries\large} {\thesection{.}}{0.2cm}{}
\titleformat{\subsection}[runin]
{\bfseries} {\thesubsection{.}}{0.15cm}{}[.]
\titleformat{\subsubsection}[runin]
{\em}{\thesubsubsection{.}}{0.15cm}{}[.]

\usepackage[up,bf]{caption}

\newtheorem{theorem}{Theorem}[section]
\newtheorem{proposition}[theorem]{Proposition}
\newtheorem{claim}[theorem]{Claim}
\newtheorem{lemma}[theorem]{Lemma}
\newtheorem{corollary}[theorem]{Corollary}

\theoremstyle{definition}
\newtheorem{definition}[theorem]{Definition}
\newtheorem{remark}[theorem]{Remark}
\newtheorem{conjecture}[theorem]{Conjecture}

\numberwithin{equation}{section}
\numberwithin{figure}{section}

\newcommand\Ocal{\mathcal{O}}

\newcommand\C{\mathbb{C}}

\newcommand\N{\mathbb{N}}

\newcommand\Z{\mathbb{Z}}

\renewcommand\b{\mathbb{B}}
\renewcommand\c{\mathbb{C}}

\newcommand\n{\mathbb{N}}
\renewcommand\r{\mathbb{R}}

\newcommand\z{\mathbb{Z}}

\newcommand\ggot{\mathfrak{g}}

\newcommand\wt{\widetilde}

\newcommand\dist{\mathrm{dist}}

\newcommand\length{\mathrm{length}}

\def\dist{\mathrm{dist}}

\def\length{\mathrm{length}}

\def\diam{\mathrm{diam}}

\usepackage{color}

\begin{document}


\fancyhead[LO]{Complete complex hypersurfaces in the ball come in foliations}
\fancyhead[RE]{A.\ Alarc\'on}
\fancyhead[RO,LE]{\thepage}

\thispagestyle{empty}


\begin{center}
{\bf\LARGE Complete complex hypersurfaces in the ball come in foliations}

\vspace*{3mm}

%
%
{\large\bf Antonio Alarc\'on}
%
\end{center}


%
%
\vspace*{3mm}

\begin{quote}
{\small
\noindent {\bf Abstract}\hspace*{0.1cm}
In this paper we prove that every smooth complete closed complex hypersurface in the open unit ball $\b_n$ of $\c^n$ $(n\ge 2)$ is a level set of a noncritical holomorphic function on $\b_n$ all of whose level sets are complete. This shows that $\b_n$ admits a nonsingular holomorphic foliation by smooth complete closed complex hypersurfaces and, what is the main point, that every hypersurface in $\b_n$ of this type can be embedded into such a foliation. We establish a more general result in which neither completeness nor smoothness of the given hypersurface is required. 

Furthermore, we obtain a similar result for
complex 
submanifolds of arbitrary positive codimension and prove the existence of a nonsingular holomorphic submersion foliation of $\b_n$ by smooth complete closed complex submanifolds of any pure codimension $q\in\{1,\ldots,n-1\}$.

\smallskip

\noindent{\bf Keywords}\hspace*{0.1cm} 
noncritical holomorphic function, 
holomorphic submersion,
Stein manifold,
complex submanifold,
complete Riemannian manifold,
divisor,
foliation,
fibre.

\smallskip

\noindent{\bf Mathematics Subject Classification (2010)}\hspace*{0.1cm} 
32H02, 
32E10, 
32E30, 
53C12. 
}
\end{quote}



\section{Introduction and main results}
\label{sec:intro}

The question whether there are {\em complete} bounded, smoothly immersed, complex submanifolds in a complex Euclidean space was asked by Yang \cite{Yang1977,Yang1977JDG} 
and affirmatively answered by Jones \cite{Jones1979PAMS} in the late 1970s. Jones constructed a complete bounded immersed complex disc in $\c^2$, an embedded one in $\c^3$, and a properly embedded one in the ball of $\c^4$. The more difficult version of this problem for {\em embedded} submanifolds of low codimension, and in particular for hypersurfaces,
has been solved only recently. It was Globevnik \cite{Globevnik2015AM} who proved that for every integer $n\ge 2$ the open unit ball $\b_n$ of $\c^n$ admits smooth complete closed complex submanifolds of any codimension $q\in\{1,\ldots,n-1\}$, thereby positively settling Yang's question for embeddings in arbitrary dimension and codimension. Globevnik's hypersurfaces $(q=1)$ are given implicitly; to be more precise, he constructed a holomorphic function $f\colon \b_n\to\c$ all of whose level sets, most of which are smooth by Sard's theorem, are complete. This shows that the ball $\b_n$ carries a (possibly singular) holomorphic foliation by complete closed complex hypersurfaces, namely, the one formed by the level sets $f^{-1}(c)$ $(c\in\c)$ of $f$. 

It is on the other hand well known that every divisor in a Stein manifold $X$ with the vanishing second cohomology group $H^2(X;\z)=0$ (as, for instance, the ball $\b_n$ of $\c^n$) is a principal divisor; in other words, every closed complex hypersurface in $X$ is a level set of a holomorphic function on $X$ (see Serre's paper \cite{Serre1953} from 1953 or Remmert  \cite[p.\ 98]{Remmert1998GTM}).
This follows from the fact that a second Cousin problem on a Stein manifold is solvable by holomorphic functions if it is solvable by continuous ones.
Much more recently Forstneri\v c \cite{Forstneric2003AM,Forstneric2018PAMS} proved that if the given hypersurface is smooth, then it admits a holomorphic defining function on $X$ which is {\em noncritical}, and hence all of its level sets are smooth closed complex hypersurfaces and they form a nonsingular holomorphic foliation of $X$.

Motivated by these seminal results, we prove in this paper that
\begin{quoting}[leftmargin={7mm}]
{\em if a closed complex hypersurface $V$ in the open unit ball $\b_n$ of $\c^n$ $(n\ge 2)$ is complete, then it is a level set of a holomorphic function $f\colon\b_n\to\c$ all of whose level sets $f^{-1}(c)$ $(c\in\c)$ are complete; if in addition $V$ is smooth, then the defining function $f$ can be chosen noncritical}.
\end{quoting}
This establishes a converse to the aforementioned Globevnik's existence theorem, namely, this proves that, actually, {\em every} complete closed complex hypersurface in $\b_n$ is defined by a holomorphic function on $\b_n$ as those constructed in \cite{Globevnik2015AM}. Moreover, our result shows that $\b_n$ admits a {\em nonsingular} holomorphic foliation by smooth complete closed complex hypersurfaces (thereby extending Globevnik's results in \cite{Globevnik2015AM}) and, what is the main point, that {\em every} complex hypersurface in $\b_n$ with these properties can be embedded into such a foliation.

We shall obtain the above result as a special case of a more general one in which completeness of the given hypersurface $V$ is not required (Theorem \ref{th:intro-q=1}); we then furnish a holomorphic defining function for $V$ on $\b_n$ all of whose level sets, except perhaps $V$ itself, are complete.
This provides extensions to the aforementioned results by Serre from \cite{Serre1953} and Forstneri\v c from \cite{Forstneric2003AM,Forstneric2018PAMS} in the case when the Stein manifold $X$ is a Euclidean ball.

Besides, we go considerably further and prove a similar result dealing also with complex submanifolds of higher codimension (Theorem \ref{th:intro-q}) and implying that for any $q\in\{1,\ldots,n-1\}$ there is a {\em nonsingular holomorphic submersion foliation} (i.e., a foliation formed by the fibres of a holomorphic submersion)  of $\b_n$ by smooth complete closed complex submanifolds of pure codimension $q$  
(Corollary \ref{co:k=0}). 



Here is a simplified version of the main result of this paper (Theorem \ref{th:MR}).
%
%
\begin{theorem}\label{th:intro-q}
Let $n$ and $q$ be integers with $1\le q<n$.
If $V$ is a smooth closed complex submanifold of pure codimension $q$ in $\b_n$ which is contained in a fibre of a holomorphic submersion from $\b_n$ to $\c^q$, then there is a holomorphic submersion $f\colon \b_n\to\c^q$ satisfying the following conditions.
\begin{itemize}
\item[\rm (i)]  $f(z)=0\in\c^q$ for all $z\in V$; hence, $V$ is a union of components of $f^{-1}(0)$.
\smallskip
\item[\rm (ii)] 
The fibre $f^{-1}(c)\subset\b_n$ is complete for every $c\in f(\b_n)\setminus\{0\}\subset\c^q$.
\smallskip
\item[\rm (iii)] $f^{-1}(0)\setminus V$ is either the empty set or a smooth complete closed complex submanifold of pure codimension $q$ in $\b_n$.
\end{itemize}
Furthermore, if $V$ is a fibre of a holomorphic submersion from $\b_n$ to $\c^q$, then $f$ can be chosen with $f(z)\neq 0$ for all  $z\in\b_n\setminus V$, and hence $f^{-1}(0)=V$.
\end{theorem}
In particular, the family of components of the fibres $f^{-1}(c)$ $(c\in \c^q)$ of the submersion $f$ furnished by the theorem is a nonsingular holomorphic submersion foliation of $\b_n$ by smooth connected closed complex submanifolds of codimension $q$ all which, except perhaps those contained in the given submanifold $V\subset f^{-1}(0)$, are complete. If $V$ is complete, then all the leaves in the foliation are complete.

%
%
There are, however, several constructions of smooth
complete closed complex submanifolds in balls, besides the implicit one.  
Jones' examples in \cite{Jones1979PAMS} were obtained by using the BMO duality technique. The first known complete properly embedded complex curves in $\b_2$, due to Alarc\'on and L\'opez \cite{AlarconLopez2016JEMS}, were found as holomorphic curves parameterized by open Riemann surfaces by means of a desingularizing argument involving a surgery which does not enable any control on the topology of the examples (this settled the embedded Yang problem for $n=2$). 
Moreover, using holomorphic automorphisms of $\c^n$ (see Anders\'en and Lempert \cite{AndersenLempert1992IM} and Forstneri\v c and Rosay \cite{ForstnericRosay1993IM}), Alarc\'on, Globevnik, and L\'opez \cite{AlarconGlobevnikLopez2019Crelle} constructed hypersurfaces of this type in $\b_n$ with restricted topology, and Alarc\'on and Globevnik \cite{AlarconGlobevnik2017C2} with arbitrary topology when $n=2$.
%
%
Some other construction techniques for submanifolds of high codimension, enabling control even on the complex structure of the examples, can be found in
Alarc\'on and Forstneri\v c \cite{AlarconForstneric2013MA} 
(based on the Riemann-Hilbert boundary value problem) and
 Drinovec Drnov\v sek \cite{Drinovec2015JMAA} (using holomorphic peak functions).
 We refer to \cite[\textsection 4.3]{Forstneric2018Survey} for a survey of results in this topic. 

Theorem \ref{th:intro-q} furnishes a holomorphic submersion from $\b_n$ to $\c^q$ having a complete fibre. 
Applying Theorem \ref{th:intro-q} with $V$ contained in such a fibre gives the following.
%
%
\begin{corollary}\label{co:k=0}
For any pair of integers $n$ and $q$ with $1\le q<n$ there is a holomorphic submersion from $\b_n$ to $\c^q$ all of whose fibres are complete. Hence, there is a nonsingular holomorphic submersion foliation of $\b_n$ by smooth complete closed complex submanifolds of pure codimension $q$.
\end{corollary}
In the special case when $q=1$, Corollary \ref{co:k=0} was proved in \cite{Globevnik2015AM} but without ensuring that the defining function of the foliation be submersive (i.e., noncritical); thus, the foliation of $\b_n$  by complex hypersurfaces may have singularities. No result in this direction is available in the literature when $q>1$ (even allowing singularities). 

Somewhat surprisingly and
contrary to the intuition that motivated this paper, it turns out that the
completeness of $V$ is not required in Theorem \ref{th:intro-q}; the connectedness of $V$ is not required, either. Thus, we also obtain the following.
%
%
\begin{corollary}\label{co:k}
For any triple of integers $n$, $q$, and $k$ with $1\le q<n$ and $k\ge 1$ there is a nonsingular holomorphic submersion foliation of $\b_n$ by smooth connected closed complex submanifolds of codimension $q$ all which, except precisely $k$ among them, are complete. 
\end{corollary}
In fact, Theorem \ref{th:intro-q} implies the more general assertion that every smooth closed complex submanifold $V$ of pure codimension $q$ in $\b_n$ consisting of precisely $k$ components, none of which is complete, can be embedded into a foliation of $\b_n$ as those in the corollary whenever $V$ lies in a fibre of a holomorphic submersion from $\b_n$  to $\c^q$. Corollary \ref{co:k} is the first result of its kind (even for $q=1$). 

The assumption in Theorem \ref{th:intro-q} that $V$ lies in a fibre of a holomorphic submersion from $\b_n$ to $\c^q$ is clearly necessary in view of property {\rm (i)} and cannot be relaxed, hence, the theorem applies to all submanifolds of the ball for which its conclusion makes sense. Note that this condition is equivalent to $V$ be a union of connected components of a fibre of such a submersion. Likewise, the assumption on $V$ in the final assertion of Theorem \ref{th:intro-q} is also necessary. Nevertheless, 
%
 the reader may be wondering when is a given smooth closed complex submanifold of pure codimension $q$ in $\b_n$ contained in a fibre of a holomorphic submersion from $\b_n$ to $\c^q$ and, more ambitiously, when is it defined by such a submersion. 
These questions, in general for an arbitrary Stein manifold in place of the ball, pertain to the classical subject of {\em complete intersections} (see e.g.\ Forster \cite{Forster1984LNM} and Schneider \cite{Schneider1982MA} for background); we refer to Forstneri\v c \cite{Forstneric2003AM,Forstneric2018PAMS} and \cite[\textsection 8.5 and \textsection 9.12-9.16]{Forstneric2017} for a discussion of the state of the art. The latter holds true, for instance, provided that $q=1$ (see \cite[Corollary 1.2]{Forstneric2018PAMS}), and so Theorem \ref{th:intro-q} applies to any smooth closed complex hypersurface. 
%
Moreover, every closed complex hypersurface $V$  (possibly with singularities) in $\b_n$ is known to admit a defining holomorphic function on $\b_n$ which is noncritical off the singular set of $V$
 (this holds true on any Stein manifold $X$ with $H^2(X;\Z)=0$; see \cite[Theorem 1.1]{Forstneric2018PAMS}). 
The following extension of the case $q=1$ in Theorem \ref{th:intro-q}  is a simplified version of the  second main result of this paper (Theorem \ref{th:MR-q=1}).
%
%
\begin{theorem}\label{th:intro-q=1}
For any closed complex hypersurface $V$ (possibly with singularities) in $\b_n$ $(n\ge 2)$ there is a holomorphic function $f$ on $\b_n$ with the following properties.
\begin{itemize}
\item[\rm (i)]  $f^{-1}(0)=V$; i.e., $f$ is a defining function for $V$. 
\smallskip
\item[\rm (ii)] The critical locus of $f$ coincides with the singular set of $V$. In particular, the function $f$ is noncritical off $V$.
\smallskip
\item[\rm (iii)] The level set $f^{-1}(c)\subset\b_n$ is complete for every $c\in f(\b_n)\setminus\{0\}\subset\c$.
\end{itemize}
\end{theorem}
It therefore turns out that {\em every} closed complex hypersurface $V\subset\b_n$ is a union of leaves in a holomorphic foliation of $\b_n$ by connected closed complex hypersurfaces (the family of components of the level sets of the function $f$ furnished by the theorem) all which, except perhaps those contained in $V=f^{-1}(0)$, are smooth and complete. If $V$ is smooth, then $f$ is noncritical and thus the foliation is nonsingular. 
If $V$ is complete, then all the leaves in the foliation are complete. 

The hypersurface $V$ in Theorem \ref{th:intro-q=1} is not assumed to be complete, connected, or smooth; the latter is the new feature with respect to Theorem \ref{th:intro-q}. On the other hand, comparing Theorem \ref{th:intro-q=1} with the aforementioned \cite[Theorem 1.1]{Forstneric2018PAMS}, what is new is condition {\rm (iii)} guaranteeing completeness of all level sets of $f\colon\b_n\to\c$, except perhaps $f^{-1}(0)=V$, and thereby ensuring that every {\em complete} closed complex hypersurface in $\b_n$ (hence, in particular, those furnished in \cite{AlarconLopez2016JEMS,AlarconGlobevnikLopez2019Crelle,AlarconGlobevnik2017C2}) is defined by a holomorphic function on $\b_n$ whose level sets are all complete; i.e., as those constructed by Globevnik in \cite{Globevnik2015AM}. Thus, as suggested in the title, every hypersurface of this type (including those with singularities) comes as a union of leaves in a holomorphic foliation of $\b_n$ by hypersurfaces of the same sort; moreover, the smooth ones come in nonsingular holomorphic foliations.

%
%
Together with the techniques in \cite{Globevnik2016MA} (a sequel to \cite{Globevnik2015AM} where Globevnik extended results to pseudoconvex domains), 
the new methods we develop in this paper enable us to obtain an analogue of Theorem \ref{th:intro-q=1} in which the role of the ball is played by an arbitrary Stein manifold equipped with a Riemannian metric. This applies in particular to any pseudoconvex domain in $\C^n$, thereby establishing a converse to Globevnik's result from \cite{Globevnik2016MA}. In Section \ref{sec:pseudo} we motivate, state, and prove this result.


%
%
\noindent{\bf Method of proof.}
Theorem \ref{th:intro-q} is proved in Section \ref{sec:proof}. The proof broadly follows the usual approach in this type of constructions; however, our argument presents some major differences and novelties. We shall provide a holomorphic submersion $f\colon \b_n\to\c^q$ satisfying the conclusion of the theorem as the limit of a sequence of holomorphic submersions $f_j\colon \b_n\to\c^q$ $(j\in\n)$. In order to guarantee that all the fibres of $f$ are complete, besides the one over $0\in\c^q$, we shall ensure that each of them is disjoint from a certain {\em labyrinth} of compact sets in $\b_n$. Several types of labyrinths have been used as a key ingredient in earlier constructions of this type (see \cite{Globevnik2015AM,Globevnik2016MA,AlarconGlobevnikLopez2019Crelle}); nevertheless, the way we build and use them in this paper is conceptually different from the one of the previous constructions (see Subsec.\ \ref{ss:l}). In the available sources one first obtains an infinite labyrinth in $\b_n$ such that every proper path $[0,1)\to\b_n$ avoiding the labyrinth has infinite length, and after that one constructs a closed complex submanifold in $\b_n$ which is disjoint from this particular labyrinth, hence, it is complete. This strategy does not seem to lead to our goal; on the contrary, our proof requires us to construct the labyrinth and the submersion $f\colon \b_n\to\c^q$ at the same time in an inductive procedure (see Lemma \ref{lem:ML}). Moreover, our labyrinth is adapted to the given submanifold, $V$, and those components of the labyrinth which intersect it need to be treated in a completely different manner 
than those which do not. In particular, it is crucial in our method that the components of the labyrinth that meet $V$ may be chosen with arbitrarily small diameter (the diameter of the rest of the components is, however, irrelevant for us). 
Despite being natural, this novel approach implies new important difficulties 
to be overcome. 

Very roughly, at each step of the induction we construct a holomorphic submersion $f_j\colon \b_n\to\c^q$, vanishing everywhere on $V$, and a labyrinth $L_j$, consisting of finitely many compact sets in a spherical shell $R_j\b_n\setminus r_j\overline\b_n\subset\c^n$ $(0<r_j<R_j<1$, $\lim_{j\to\infty}r_j=1)$, in such a way that more and more fibres of the limit submersion $f$ avoid the infinite labyrinth $\bigcup_{i\ge j} L_i$ as $j\to\infty$. For that, we ensure that $|f|$ takes big values on the components of $L_j$ which are disjoint from $V$ and small ones on those which are not. At the same time we guarantee that $f^{-1}(0)\setminus V$, if nonempty, does not intersect any $L_j$. The labyrinth $L_j$ depends on both the submersion $f_{j-1}$ obtained in the previous step of the induction and the given submanifold $V$. 
In order to carry out this program we use in a strong way that the union of the ball $r_j\overline\b_n$  and our finite labyrinth $L_j$ is a {\em polynomially convex} compact set in $\c^n$ (see Remark \ref{rem:pol}).

The main tool in the proof is a rather sophisticated Oka-Weil-Cartan type theorem for holomorphic {\em submersions} from a Stein manifold $X$ to $\c^q$ $(1\le q<\dim X)$ which was furnished by Forstneri\v c in \cite{Forstneric2003AM,Forstneric2018PAMS} (see also \cite[\textsection 9.12-9.16]{Forstneric2017}). This is another substantial difference with respect to earlier constructions since, thus far, only the standard Runge approximation for functions had been involved in the analysis of this subject (see \cite{AlarconLopez2016JEMS,Globevnik2015AM,Globevnik2016MA,AlarconGlobevnikLopez2019Crelle,Globevnik2016JMAA,AlarconGlobevnik2017C2}). If we are given a smooth closed complex submanifold $V\subset X$ which is defined as the zero fibre of a holomorphic submersion $h\colon U\to\c^q$ on a neighborhood $U$ of $V$ in $X$, then the mentioned result
enables us to approximate $h$ uniformly on holomorphically convex compact subsets by holomorphic submersions $\wt h\colon X\to\c^q$ with $\wt h(x)= 0\in\c^q$ for all $x\in V$, provided that there is a {\em $q$-coframe} on $X$ 
agreeing with the differential $dh$ on $U$ (see Subsec.\ \ref{ss:COW} for terminology and a precise statement). An additional difficulty in the proof of Theorem \ref{th:intro-q} is, therefore, the need of having  at hand suitable $q$-coframes on the ball $\b_n$ at each step in the construction procedure (see Claim \ref{cl:ohi}).

The proof of Theorem \ref{th:intro-q=1}, which is similar to the one of Theorem \ref{th:intro-q}, is sketched in Section \ref{sec:q=1}. 


\section{Preliminaries}\label{sec:prelim}

We denote by $|\cdot|$, $\dist(\cdot,\cdot)$, $\length(\cdot)$, and $\diam(\cdot)$ the Euclidean norm, distance, length, and diameter in $\r^n$ for any $n\in\n=\{1,2,3,\ldots\}$; also $\z_+=\n\cup\{0\}$. Let $\b_n=\{z\in\c^n\colon |z|<1\}$ denote the open unit ball in $\c^n$ for $n\ge 2$.  

As it is customary, given a set $A$ in a topological space $X$ we denote by $\overline A$ and $\mathring A$ the topological closure and interior of $A$ in $X$, respectively, and $bA=\overline A\setminus \mathring A$. If $B\subset X$, then we write $A\Subset B$ when $\overline A\subset\mathring B$. Assume that $X$ is a Stein manifold. We shall say that a function $A\to\c$ is {\em holomorphic} if it is holomorphic in an unspecified open set in $X$ containing $A$; we denote by $\Ocal(A)$ the algebra of all such functions. A compact set $A\subset X$ is said to be {\em holomorphically convex} in $X$ (or, shortly, {\em $\Ocal(X)$-convex}) if for each $x\in X\setminus A$ there is $f\in \Ocal(X)$ with $|f(x)|>\max\{|f(a)|\colon a\in A\}$; if this is the case, then the Oka-Weil theorem (see \cite[Theorem 2.3.1]{Forstneric2017}; see \cite[Theorem 1.5.1]{Stout2007PM} for the $X=\c^n$) ensures that every function in $\Ocal(A)$ can be approximated uniformly on $A$ by functions in $\Ocal(X)$. Holomorphically convex compact sets in $\c^n$ are called {\em polynomially convex}; we refer to Stout \cite{Stout2007PM} for a monograph on the subject. 

A closed complex subvariety (i.e., embedded, possibly with singularities) $V$ in a domain $\Omega\subset\c^n$ is said to be {\em complete} if every proper (divergent) path $\gamma \colon [0,1)\to\Omega$ with $\gamma([0,1))\subset V$ has infinite Euclidean length. 
Completeness of $V$ is equivalent to the (possibly singular) Riemannian metric $\ggot$ induced on $V$ by the Euclidean one in $\c^n$ be complete in the classical sense; i.e., every component of $(V,\ggot)$ is a {complete metric space} meaning that Cauchy sequences are convergent. If $V$ is smooth and complete then, by the Hopf-Rinow theorem, it is also {geodesically complete}. Completeness of a closed complex subvariety in a Stein manifold equipped with a Riemannian metric is defined in the same way. We refer, for instance, to do Carmo \cite{doCarmo1992} for an introduction to Riemannian geometry.

%
%
\subsection{An Oka-Weil-Cartan theorem for holomorphic submersions}\label{ss:COW}

Let $X$ be a Stein manifold. A holomorphic map $f=(f_1,\ldots,f_q)\colon X\to\c^q$ is said to be {\em submersive} at a point $x\in X$ if its differential $df_x\colon T_xX\to T_{f(x)}\c^q\cong \c^q$ is surjective; equivalently, if the differentials of the component functions of $f$ at $x$ are linearly independent: 
$
	(df_1\wedge\cdots\wedge df_q)|_x\neq 0.
$ 
 The map $f$ is said to be a {\em submersion} if it is submersive at every point of $X$. If $f\colon X\to\c^q$ is a holomorphic submersion then the fibres $f^{-1}(c)$ $(c\in\c^q)$ of $f$ form a nonsingular holomorphic foliation of $X$ by smooth closed complex submanifolds of pure codimension $q$.

Following the terminology in \cite[p.\ 148]{Forstneric2003AM}, a {\em $q$-coframe} on $X$ is a $q$-tuple of continuous differential $(1,0)$-forms $\theta=(\theta_1,\ldots,\theta_q)$ on $X$ which are pointwise linearly independent at every point of $X$.
It is clear that if $h=(h_1,\ldots,h_q)\colon X\to\c^q$ is a holomorphic submersion, then the differential $dh=(dh_1,\ldots,dh_q)$ is a $q$-coframe on $X$; in this case $\theta=dh$ is said to be exact holomorphic. In the opposite direction, if $X$ admits a $q$-coframe for some integer $q$ with  $1\le q<\dim X$, then $X$ carries also a holomorphic submersion to $\c^q$ (see  \cite[Theorem 2.5]{Forstneric2003AM}).
A compilation of results by Forstneri\v c from \cite{Forstneric2003AM,Forstneric2018PAMS} (see also \cite[\textsection 9.13]{Forstneric2017}), together with inspections of their proofs, gives the following Oka-Weil-Cartan type theorem which will be exploited in the proof of Theorem \ref{th:intro-q}; we state it here for later reference.

\begin{theorem}[Forstneri\v c-Oka-Weil-Cartan theorem for holomorphic submersions]
\label{th:Divisor}
Let $X$ be a Stein manifold, $K\subset X$ be an $\Ocal(X)$-convex compact set, and $V\subset X$ be a closed complex subvariety. Also let $q\in\{1,\ldots,\dim X-1\}$ and assume that there is a holomorphic submersion $h\colon U\to\c^q$  on a neighborhood $U$ of $K\cup V$ in $X$. If there is a $q$-coframe $\theta$ on $X$ such that $\theta|_U=dh$, then for any  $\epsilon>0$ and $s\in\n$ there is a holomorphic submersion $\wt h\colon X\to\c^q$ such that $|\wt h(x)-h(x)|<\epsilon$ for all $x\in K$ and $\wt h$ agrees with $h$ to order $s$ on $V$.

Furthermore, if in addition $h|_V\equiv 0$ and $h$ admits a continuous extension to $X$ with $h(x)\neq 0$ for all $x\in X\setminus V$, then $\wt h$ can be chosen with $\wt h^{-1}(0)=V$.
\end{theorem}
The only novelty in the statement of Theorem \ref{th:Divisor} with respect to the combination of \cite[Theorem 2.5]{Forstneric2003AM} and \cite[Theorem 1.3]{Forstneric2018PAMS} is that the approximation and the non-vanishing conditions satisfied by $\wt h$ hold true simultaneously. This addition is clearly ensured by an inspection of the proofs in the cited sources. 

Since the ball $\b_n\subset\c^n$ is a Stein manifold, Theorem \ref{th:Divisor} applies with $X=\b_n$. 

%
%
\subsection{Labyrinths}\label{ss:l}
We shall use the labyrinths of compact sets in spherical shells that were introduced 
in \cite{AlarconGlobevnikLopez2019Crelle}. 
%
%
\begin{definition}\label{def:labyrinth}
Let $n\ge 2$ be an integer and let $0<r<R$. We shall say that a compact set $L$ in the spherical shell $R\b_n\setminus r\overline\b_n=\{z\in\c^n\colon r<|z|<R\}$ is a {\em tangent labyrinth} if $L$ has finitely many connected components, $T_1,\ldots,T_k$ $(k\in\n)$, and $L$ is the support of a tidy collection of tangent balls in the sense of \cite[Def.\ 1.3 and 1.4]{AlarconGlobevnikLopez2019Crelle}; i.e., $L$ satisfies the following conditions:
\begin{itemize}
\item Each $T_j$ is a closed round ball in a real affine hyperplane in $\c^n\equiv\r^{2n}$ which is orthogonal to the position vector of the center $x_j$ of the ball $T_j$. 
\smallskip
\item If $|x_i|=|x_j|$ for some $\{i,j\}\subset\{1,\ldots,k\}$, then the radii of $T_i$ and $T_j$ are equal.
 If $|x_i|<|x_j|$ for some $\{i,j\}\subset\{1,\ldots,k\}$, then $T_i\subset |x_j|\b_n$.
\end{itemize}
\end{definition}

By \cite[Lemma 2.4]{AlarconGlobevnikLopez2019Crelle}, for any numbers $0<r<R$ and $\delta>0$ there is a tangent labyrinth $L$ in $R\b_n\setminus r\overline\b_n$ for which if $\gamma\colon[0,1]\to \c^n$ is a path crossing the spherical shell $R\b_n\setminus r\overline\b_n$ from one side to the other and avoiding $L$, then the length of $\gamma$ is greater than $\delta$. 
An elementary observation (but crucial in our construction) is that there is such a tangent labyrinth with the additional property that the diameter of each of its components is as small as desired. 
%
%
\begin{lemma}
\label{lem:labyrinth}
For any numbers $0<r<R$, $\delta>0$, and $\eta>0$ there is a tangent labyrinth $L$ in $R\b_n\setminus r\overline\b_n$ $(n\ge 2)$ satisfying the following conditions.
\begin{itemize}
\item[\rm (i)] If $\gamma\colon[0,1]\to \c^n$ is a path with $|\gamma(0)|\le r$, $|\gamma(1)|\ge R$, and $\gamma([0,1])\cap L=\varnothing$, then $\length(\gamma)>\delta$.
\smallskip
\item[\rm (ii)] $\diam(T)<\eta$ for every connected component $T$ of $L$.
\end{itemize}
\end{lemma}
\begin{proof}
Pick a positive number $R_0$ such that 
\begin{equation}\label{eq:R'}
	r< R_0<\min\Big\{R\,,\, \sqrt{r^2+\frac{\eta^2}4}\Big\}.
\end{equation}
Applying \cite[Lemma 2.4]{AlarconGlobevnikLopez2019Crelle} to the spherical shell $R_0\b_n\setminus r\overline\b_n$ and the number $\delta>0$ we obtain a tangent labyrinth $L$ in $R_0\b_n\setminus r\overline\b_n$ for which condition {\rm (i)} is satisfied. Since $r<R_0<R$ and $L$ is a tangent labyrinth in $R_0\b_n\setminus r\overline\b_n$, the compact set $L$ is also a tangent labyrinth in the larger spherical shell $R\b_n\setminus r\overline\b_n$. On the other hand, since each component $T$ of $L$ is contained in a real affine hyperplane in $\r^{2n}\equiv\c^n$ and $T\subset R_0\b_n\setminus r\overline\b_n$, Pythagoras' theorem gives
\[
	\diam(T)<2\sqrt{R_0^2-r^2}.
\]
In view of the second inequality in \eqref{eq:R'}, this ensures condition {\rm (ii)}.
\end{proof}

It follows from Definition \ref{def:labyrinth} that if $T_1,\ldots, T_k$ are the components of a tangent labyrinth $L$ in $R\b_n\setminus r\overline\b_n$, then, up to a reordering, there are compact (geometrically) convex subsets $C_1\Subset \cdots\Subset C_k$ in $R\b_n$ satisfying $r\overline\b_n\cup\bigcup_{i=1}^j T_i\subset\mathring C_j$ for all $j\in\{1,\ldots,k\}$ and $C_j\cap \bigcup_{i=j+1}^k T_i=\varnothing$ for all $j\in\{1,\ldots,k-1\}$. Recall on the other hand that, by Kallin's lemma (see \cite{Kallin1965} or \cite[p.\ 62]{Stout2007PM}) and the Oka-Weil theorem (see \cite[Theorem 1.5.1]{Stout2007PM}), if $C$ and $T$ are two disjoint compact convex sets in $\c^n$ and $K\subset C$ is a compact polynomially convex set, then the union $K\cup T$ is polynomially convex as well. 
A straightforward finite recursive application of this fact shows the following property of tangent labyrinths which is pointed out in \cite{AlarconForstneric2017PAMS}.
%
%
\begin{remark}\label{rem:pol}
If $L$ is a tangent labyrinth  in $R\b_n\setminus r\overline\b_n$ $(0<r<R$, $n\ge 2)$, then the compact set $r\overline\b_n\cup L\subset\c^n$ is polynomially convex.
\end{remark}

\section{Proof of Theorem \ref{th:intro-q}}\label{sec:proof}

The following more precise version of Theorem \ref{th:intro-q}, including approximation, interpolation, and certain control on the growth of the length of divergent curves within a fibre, is the first main result of this paper.
%
%
\begin{theorem}\label{th:MR}
Let $n$ and $q$ be integers with $1\le q<n$.
Let $V$ be a smooth closed complex submanifold of pure codimension $q$ in $\b_n$ and assume that there are a holomorphic submersion $h\colon U\to\c^q$ on a neighborhood $U$ of $V$ in $\b_n$ and a $q$-coframe $\theta$ on $\b_n$ (see Subsec.\ \ref{ss:COW}) such that 
\begin{equation}\label{eq:VhtUh}
	V\subset h^{-1}(0)\quad \text{and}\quad \theta|_U=dh.
\end{equation}
Choose a polynomially convex compact set $K\subset U$ and an increasing sequence $0<r_1<R_1<r_2<R_2<\cdots$ with 
$K\subset r_1\b_n$ and $\lim_{j\to\infty} r_j=1$. 
Then, for any number $\epsilon>0$ and any increasing sequence $0<\delta_1<\delta_2<\cdots$ there is a holomorphic submersion $f\colon \b_n\to\c^q$ satisfying the following conditions.
\begin{itemize}
\item[\it i)] $f(z)=0$ for all $z\in V$ and $f$ agrees with $h$ to any given finite order on $V$. Hence, $V$ is a union of components of the fibre $f^{-1}(0)$ of $0\in\c^q$ under $f$.
\smallskip
\item[\it ii)] $|f(z)-h(z)|<\epsilon$ for all $z\in K$.
\smallskip
\item[\it iii)] For any $\lambda>0$ there is $j_\lambda\in\n$ for which if $\gamma\colon[0,1]\to \b_n$ is a path such that $|\gamma(0)|\le r_j$ and $|\gamma(1)|\ge R_j$ for some $j\ge j_\lambda$ and 
\[
	\lambda\le|f(\gamma(t))|\le \frac1{\lambda}\quad \text{for all $t\in [0,1]$},
\]
then $\length (\gamma)>\delta_j$. 
In particular, every proper path $[0,1)\to\b_n$ on which $|f|$ is bounded above and bounded away from zero has infinite length.
\smallskip
\item [\it iv)] If $f^{-1}(0)\setminus V\neq\varnothing$ and $\gamma\colon[0,1]\to f^{-1}(0)\setminus V$ is a path such that $|\gamma(0)|\le r_j$ and $|\gamma(1)|\ge R_j$ for some $j\in\n$, then $\length (\gamma)>\delta_j$. In particular, every proper path $[0,1)\to\b_n$ whose image lies in $f^{-1}(0)\setminus V$ has infinite length.
\end{itemize}
Furthermore, if $h$ admits a continuous extension to $\b_n$ with $h(z)\neq 0\in\c^q$ for all $z\in\b_n\setminus V$, then the submersion $f$ can be chosen with $f^{-1}(0)=V$.
\end{theorem}
Concerning the assumptions on the submanifold $V$ in the theorem recall that, by the Docquier and Grauert tubular neighborhood theorem  \cite{DocquierGrauert1960MA}, if $V$ is a smooth closed complex submanifold of pure codimension $q$ in $\b_n$, then there are an open neighborhood $U$ of $V$ in $\b_n$ and a holomorphic submersion $h\colon U\to\c^q$ such that $V\subset h^{-1}(0)$ if and only if the normal bundle $N_{V/\b_n}=T\b_n|_V/TV$ of $V$ in $\b_n$ is trivial.
Together with the additional condition that there is a $q$-coframe $\theta$ on $\b_n$ such that $\theta|_U=dh$, this implies that $V$ is a union of connected components of a fibre of a holomorphic submersion from $\b_n$  to $\c^q$, which is equivalent to the assumption on $V$ in Theorem \ref{th:intro-q}. (This remains true with $\b_n$ replaced by any Stein manifold of dimension greater than $q$; see \cite{Forstneric2003AM}.)

%
%
%

%
%
\begin{proof}
We begin with the following reduction which, in view of the assumptions on $h$ and $\theta$ in \eqref{eq:VhtUh}, is ensured by Theorem \ref{th:Divisor}.
%
%
\begin{claim}\label{cl:f0}
We may assume without loss of generality that $h$ extends to a holomorphic submersion  
\begin{equation}\label{eq:f0}
f_0\colon\b_n\to\c^q.
\end{equation} 
In particular, $f_0^{-1}(0)\cap U=h^{-1}(0)\supset V$.
Furthermore, if $h$ admits a continuous extension to $\b_n$ with $h(z)\neq 0\in\c^q$ for all $z\in\b_n\setminus V$, then we may assume in addition that $f_0(z)\neq 0$ for all $z\in\b_n\setminus V$ and hence $f_0^{-1}(0)=V$.
\end{claim}

Call $\epsilon_0=\epsilon$ and let $1>\lambda_1>\lambda_2>\cdots$ be a decreasing sequence of positive numbers  with 
\begin{equation}\label{eq:llj0}
	\lim_{j\to\infty}\lambda_j=0.
\end{equation} 
Also set $L_0=\varnothing$ and $O_0=\varnothing$.

A submersion $f\colon \b_n\to\c^q$ satisfying the conclusion of the theorem shall be obtained in an inductive way as the limit of a sequence of holomorphic submersions $f_j\colon\b_n\to\c^q$ $(j\in\z_+)$; see \eqref{eq:f}. The basis of the induction shall be given by the already fixed submersion $f_0$ in \eqref{eq:f0}. 

The main step in the proof is enclosed in the following lemma.
%
%
\begin{lemma}\label{lem:ML}
There is a sequence $S_j=\{f_j,\epsilon_j,L_j,O_j\}$ $(j\in\n)$, where
\begin{itemize}
\item $f_j\colon \b_n\to\c^q$ is a holomorphic submersion,
\smallskip
\item $\epsilon_j>0$ is a number,
\smallskip
\item $L_j$ is a tangent labyrinth in $R_j\b_n\setminus r_j\overline \b_n$ (see Definition \ref{def:labyrinth}), and
\smallskip
\item $O_j$ is an open neighborhood of $L_j$ in $R_j\b_n\setminus r_j\overline \b_n$,
\end{itemize}
 such that the following conditions are satisfied for all $j\in\n$.
\begin{itemize}
\item[\rm (1$_j$)] $|f_j(z)-f_{j-1}(z)|<\epsilon_j$ for all $z\in r_j\overline\b_n$.
\smallskip
\item[\rm (2$_j$)] $f_j(z)=0\in\c^q$ for all $z\in V$  and $f_j$ agrees with $h$ to the given finite order everywhere on $V$.
\smallskip
\item[\rm (3$_j$)] If $f_0^{-1}(0)=V$, then $f_j(z)\neq 0$ for all $z\in\b_n\setminus V$.
\smallskip
\item[\rm (4$_j$)] $\displaystyle 0<\epsilon_j<\epsilon_{j-1}/2$.
\smallskip
\item[\rm (5$_j$)] If $\varphi\colon \b_n\to\c^q$ is a holomorphic map such that $|\varphi(z)-f_{j-1}(z)|<2\epsilon_j$ for all $z\in r_j\overline\b_n$, then $\varphi$ is submersive everywhere in $R_{j-1}\overline\b_n$.
\smallskip
\item[\rm (6$_j$)] If $\gamma\colon [0,1]\to \b_n$ is a path such that $|\gamma(0)|\le r_j$, $|\gamma(1)|\ge R_j$, and $\gamma([0,1])\cap L_j=\varnothing$, then $\length(\gamma)> \delta_j$.
\smallskip
\item[\rm (7$_j$)] If $z\in L_j$, then either $|f_j(z)|<\lambda_j$ or $|f_j(z)|>1/\lambda_j$.
\smallskip
\item[\rm (8$_j$)] $f_j^{-1}(0)\cap O_i\setminus V=\varnothing$ for all $i\in\{0,\ldots,j\}$.
\end{itemize}
\end{lemma}
%
%
\begin{proof}
We proceed by induction. The basis is provided by the already given holomorphic submersion $f_0\colon\b_n\to\c^q$, which has been fixed in Claim \ref{cl:f0}, together with $\epsilon_0=\epsilon$, $L_0=\varnothing$, and $O_0=\varnothing$. Recall that $f_0|_U=h$ and if in addition $h$ extends to a continuous map from $\b_n$ to $\c^q$ not assuming the value $0\in\c^q$ on $\b_n\setminus V$, then $f_0^{-1}(0)=V$. This and the assumptions on $h$ in the statement of the theorem imply {\rm (2$_0$)}. The rest of the conditions, except for {\rm (3$_0$)}, {\rm (7$_0$)}, and {\rm (8$_0$)} which trivially hold true, are vacuous for $j=0$.

In order to prove the inductive step fix $j\in\n$, assume that we already have tuples $S_i=\{f_i,\epsilon_i,L_i,O_i\}$ for all $i\in\{0,\ldots, j-1\}$, meeting the corresponding requirements, and let us provide a suitable tuple $S_j$. 

Since $f_{j-1}\colon \b_n\to\c^q$  is a holomorphic submersion and $R_{j-1}\overline\b_n\Subset r_j\overline\b_n$ are compact, the Cauchy estimates give a number $\epsilon_j>0$ satisfying {\rm (4$_j$)} and {\rm (5$_j$)}. 

By {\rm (2$_{j-1}$)}, which is satisfied even when $j=1$, we have that $f_{j-1}(z)=0\in\c^q$ for all $z\in V$. Therefore, since $f_{j-1}$ is continuous on $\b_n$ and $R_j\overline \b_n\subset \b_n$ is compact, there is a number $\eta>0$ such that
\begin{equation}\label{eq:tau1}
	|f_{j-1}(z)|<\lambda_j \quad \text{for all $z\in R_j\overline\b_n$ with $\dist(z,V)<\eta$.}
\end{equation}
Moreover, since $f_{j-1}\colon\b_n\to\c^q$ is a holomorphic submersion and $V\subset \b_n$ is a smooth complex submanifold of pure codimension $q$ in $\b_n$ which is contained in $f_{j-1}^{-1}(0)$, we have that $V$ is an open subset of $f_{j-1}^{-1}(0)$. Taking into account that $V$ is closed in $\b_n$, and hence in $f_{j-1}^{-1}(0)$, we infer that $V$ is a union of connected components of the fibre $f_{j-1}^{-1}(0)$ of $0\in\c^q$ under $f_{j-1}$. Therefore,
$f_{j-1}$ defines $V$ as a complete intersection on a neighborhood of $V$ in $\b_n$; i.e., $V$ is the fibre of $0$ under the restriction of $f_{j-1}$ to a neighborhood of $V$ in $\b_n$. Thus, by compactness of $R_j\overline\b_n\subset\b_n$ and up to passing to a smaller $\eta>0$ if necessary, we also have that
\begin{equation}\label{eq:tau2}
	f_{j-1}(z)\neq 0 \quad \text{for all $z\in R_j\overline\b_n\setminus V$ with $\dist(z,V)<\eta$.}
\end{equation}

With the number $\eta>0$ for which \eqref{eq:tau1} and \eqref{eq:tau2} hold true at hand, we apply Lemma \ref{lem:labyrinth} in order to obtain a tangent labyrinth $L_j$ in the spherical shell $R_j\b_n\setminus r_j\overline\b_n$ satisfying the following properties.
\begin{itemize}\label{eq:lab}
\item[\rm (A)]  If $\gamma\colon[0,1]\to \b_n$ is a path such that $|\gamma(0)|\le r_j$, $|\gamma(1)|\ge R_j$, and $\gamma([0,1])\cap L_j=\varnothing$, then $\length(\gamma)>\delta_j$. This is condition {\rm (6$_j$)}.
\smallskip
\item[\rm (B)] The diameter of every connected component of $L_j$ is smaller than $\eta$.
\end{itemize} 

We shall now split the tangent labyrinth $L_j$ into two pieces.
On the one hand, let $\Lambda_V$ denote the union of all the connected components of $L_j$ which have nonempty intersection with $V$. It turns out, by {\rm (B)}, that  
$\dist(z,V)<\eta$ for all $z\in \Lambda_V$; hence, in view of \eqref{eq:tau1} and \eqref{eq:tau2} we have that
\begin{equation}\label{eq:L1}
	0< |f_{j-1}(z)|<\lambda_j \quad \text{for all $z\in \Lambda_V\setminus V$}.
\end{equation}
On the other hand, denote by $\Lambda_0$ the union of all the connected components of $L_j$ which are disjoint from $V$: $\Lambda_0=L_j\setminus \Lambda_V$. Note that either $\Lambda_V$ or $\Lambda_0$ could be empty and, obviously, $L_j=\Lambda_V\cup\Lambda_0$ and $\Lambda_V\cap\Lambda_0=\varnothing$.

The next step in the proof consists in suitably deforming $f_{j-1}$ by means of the Forstneri\v c-Oka-Weil-Cartan theorem for holomorphic submersions (Theorem \ref{th:Divisor}). The precise initial objects for this deformation are provided by the following result.
\begin{claim}\label{cl:ohi}
There are an open neighborhood $W$ of $r_j\overline\b_n\cup V\cup L_j$ in $\b_n$ and a holomorphic submersion $\phi\colon W\to\c^q$ satisfying the following conditions.
\begin{itemize}
\item[\rm (a)] $\phi(z)=f_{j-1}(z)$ for all $z\in r_j\overline\b_n\cup V\cup \Lambda_V$.
\smallskip
\item[\rm (b)] $|\phi(z)|>1/\lambda_j$  for all $z\in\Lambda_0$.
\smallskip
\item[\rm (c)] There is a $q$-coframe $\theta_j=(\theta_{j,1}, \ldots, \theta_{j,q})$ on $\b_n$ such that $\theta_j|_W=d\phi$. 
\smallskip
\item[\rm (d)] If $f_{j-1}^{-1}(0)=V$, then $\phi$ admits a continuous extension to $\b_n$ with $\phi(z)\neq 0$ for all $z\in\b_n\setminus V$. 
\end{itemize}
\end{claim}
By the identity principle, condition {\rm (a)} forces that $\phi(z)=f_{j-1}(z)$ for all $z$ in the component of $W$ containing $r_j\overline\b_n\cup V\cup \Lambda_V$.
\begin{proof}
Denote by $T_1,\ldots, T_k$ $(k\in\n)$ the connected components of the compact set $\Lambda_0\subset R_j\b_n\setminus (r_j\overline\b_n\cup V\cup\Lambda_V)$. Since $\Lambda_0$ is a union of components of the tangent labyrinth $L_j$, we have that each set $T_a$ is a closed round ball in a real affine hyperplane in $\c^n\equiv\r^{2n}$ which is orthogonal to the position vector of the center of the ball (see Definition \ref{def:labyrinth}). In particular, each $T_a$ is homeomorphic to the closed unit ball in  the real Euclidean space $\r^{2n-1}$. 
Consider the compact convex sets
\[
	T_{a,\sigma}=\{z\in\c^n\colon \dist(z,T_a)\le \sigma \mu\}, \quad  
	a=1,\ldots,k,\; \sigma=1,2,
\] 
where $\mu>0$ is so small that
the following conditions are satisfied:
\begin{itemize}
\item[\rm (i)] $T_{a,2}\subset R_j\b_n\setminus (r_j\overline\b_n\cup V\cup \Lambda_V)$ for all $a\in\{1,\ldots,k\}$.
\smallskip
\item[\rm (ii)] The sets $T_{1,2},\ldots, T_{k,2}$ are pairwise disjoint.
\end{itemize}
A number $\mu>0$ with these properties exists since the set $\Lambda_0\subset R_j\b_n\setminus r_j\overline\b_n$ is compact and has finitely many connected components; take also into account that $V$ is closed in $\b_n$, $\Lambda_V$ is compact, and $\Lambda_0\cap (V\cup\Lambda_V)=\varnothing$. 

Obviously, $T_{a,1}\subset\mathring T_{a,2}$ for all $a\in\{1,\ldots,k\}$. 
Set
\[
	\Delta_\sigma= \bigcup_{a=1}^k T_{a,\sigma},\quad \sigma=1,2,
\]
and note that {\rm (i)} ensures that
\begin{equation}\label{eq:W}
	\big(r_j\overline\b_n\cup V\cup \Lambda_V\big)\cap \Delta_2=\varnothing.
\end{equation}
Set
\begin{equation}\label{eq:Wdef}
	W=( \b_n\setminus \Delta_2) \cup \mathring \Delta_1=
	\b_n\setminus\bigcup_{a=1}^k (T_{a,2}\setminus \mathring T_{a,1})
\end{equation}
and observe that $W$ is an open neighborhood of $r_j\overline\b_n\cup V\cup L_j$ in $\b_n$. Indeed, taking into account \eqref{eq:W}, {\rm (ii)}, and the closeness of $\Delta_2$ in $\b_n$ we infer that $\b_n\setminus\Delta_2$ is a connected open neighborhood of $r_j\overline\b_n\cup V\cup \Lambda_V$. On the other hand, we have 
\begin{equation}\label{eq:lambda01}
	\Lambda_0\subset \mathring \Delta_1,
\end{equation}
and hence $\Delta_1$ is an open neighborhood of $\Lambda_0$. Since $L_j=\Lambda_V\cup\Lambda_0$, this shows that  $W$ is an open neighborhood of $r_j\overline\b_n\cup V\cup L_j$, as claimed.

We distinguish cases.

\noindent{\em Case 1.} Suppose that $f_{j-1}^{-1}(0)\setminus V\neq\varnothing$. 
Choose $w_0\in\c^q$ such that 
\begin{equation}\label{eq:w0}
	|f_{j-1}(z)+w_0|>\frac1{\lambda_j}\quad \text{for all $z\in \Lambda_0$}; 
\end{equation}
any $w_0$ with large enough norm meets this requirement since $f_{j-1}(\Lambda_0)\subset\c^q$ is compact. Define
\[
	W\ni z\longmapsto \phi(z)=\left\{
	\begin{array}{ll}
	f_{j-1}(z) & \text{if }z\in \b_n\setminus\Delta_2 \medskip	
	\\
	f_{j-1}(z)+w_0 & \text{if }z\in \mathring \Delta_1.
	\end{array}
	\right.
\]
By {\rm (ii)}, \eqref{eq:W}, \eqref{eq:Wdef}, \eqref{eq:lambda01}, \eqref{eq:w0}, and the facts that $f_{j-1}$ is a holomorphic submersion and that $w_0$ is a constant, it turns out that $\phi\colon W\to\c^q$ is a well defined holomorphic submersion  and satisfies conditions {\rm (a)} and {\rm (b)}.
Since we are assuming that $f_{j-1}^{-1}(0)\setminus V\neq\varnothing$, condition {\rm (d)} is trivially satisfied; thus, to finish the proof it only remains to find a $q$-coframe $\theta_j$ on $\b_n$ satisfying {\rm (c)}. Since $w_0\in\c^q$ is a constant, we have that $d\phi=(d f_{j-1})|_W$, and hence it is clear that the exact holomorphic $q$-coframe $\theta_j=df_{j-1}$ on $\b_n$ meets the required condition.

\smallskip
\noindent{\em Case 2.} Suppose that $f_{j-1}^{-1}(0)=V$.  Since $\Lambda_0\cap V=\varnothing$, we have in this case that $f_{j-1}(z)\neq 0\in\c^q$ for all $z$ in the compact set $\Lambda_0$, and hence there is $C>1$ such that 
\begin{equation}\label{eq:C>}
	C\,|f_{j-1}(z)|>\frac1{\lambda_j}\quad \text{for all $z\in \Lambda_0$}. 
\end{equation}
Consider a continuous map 
$
	\rho\colon \Delta_2\setminus\mathring \Delta_1 =\bigcup_{a=1}^k (T_{a,2}\setminus\mathring T_{a,1})\to [1,C]
$
such that 
\begin{equation}\label{eq:varrho}
	\left\{
	\begin{array}{ll}
	\rho(z)=C & \text{for all }z\in b\Delta_1=\bigcup_{a=1}^k bT_{a,1}\medskip
	\\
	\rho(z)=1 & \text{for all }z\in b\Delta_2=\bigcup_{a=1}^k bT_{a,2}.
	\end{array}
	\right.
\end{equation} 
Existence of such a map is clear since $\Delta_2\setminus\mathring \Delta_1$ is homeomorphic to a finite union of pairwise disjoint closed spherical shells in $\c^n$ and its boundary coincides with $b\Delta_1\cup b\Delta_2$, which is a disjoint union. Define
\[
	\b_n\ni z\longmapsto \phi(z)=\left\{
	\begin{array}{ll}
	f_{j-1}(z) & \text{if }z\in \b_n\setminus \Delta_2 \medskip
	\\
	\rho(z)f_{j-1}(z) & \text{if }z\in \Delta_2\setminus \mathring \Delta_1 \medskip
	\\
	C\, f_{j-1}(z) & \text{if }z\in \mathring  \Delta_1.
	\end{array}
	\right.
\]
By conditions {\rm (i)}, {\rm (ii)}, and \eqref{eq:varrho}, the map $\phi\colon\b_n\to\c^q$ is well defined and continuous. Since $\rho$ has no zeros (recall that it takes values in the interval $[1,C]$) and $C\neq 0$, we have that $\phi^{-1}(0)=f_{j-1}^{-1}(0)=V$; this shows {\rm (d)}. Taking into account \eqref{eq:Wdef} and that $C\neq 0$ is a constant, we infer that $\phi|_W\colon W\to\c^q$ is a holomorphic submersion since so is $f_{j-1}\colon\b_n\to\c^q$. 
(Possibly, $\phi$ is neither holomorphic nor submersive on $\Delta_2\setminus \mathring \Delta_1=\b_n\setminus W$; we do not take care of that.)
On the other hand, it is clear from \eqref{eq:W}, \eqref{eq:lambda01}, and \eqref{eq:C>} that $\phi$ satisfies the requirements {\rm (a)} and {\rm (b)} in the statement of the claim.
To finish, set
\[
	\theta_j=(\theta_{j,1},\ldots,\theta_{j,q})=\left\{
	\begin{array}{ll}
	df_{j-1} & \text{on } \b_n\setminus \Delta_2 \medskip
	\\
	\rho\, df_{j-1} & \text{on } \Delta_2\setminus \mathring \Delta_1 \medskip
	\\
	C\, df_{j-1} & \text{on }\mathring \Delta_1.
	\end{array}
	\right.
\]
It follows from \eqref{eq:varrho} that $\theta_j$ is a continuous differential $(1,0)$-form on $\b_n$. Moreover, since $\rho$ has no zeros, $C\neq 0$, and $f_j$ is a holomorphic submersion on $\b_n$, it turns out that the components $\theta_{j,1},\ldots,\theta_{j,q}$ of $\theta_j$ are linearly independent at every point of $\b_n$; i.e., $\theta_j$ is a $q$-coframe on $\b_n$. Furthermore, taking into account \eqref{eq:Wdef} and that $C$ is a constant, we infer that $\theta_j|_W$ is exact holomorphic and agrees with $d(\phi|_W)$. This shows condition {\rm (c)}.

Claim \ref{cl:ohi} is proved.
\end{proof}

We continue the proof of Lemma \ref{lem:ML}. Fix a real number $\epsilon'$, $0<\epsilon'<\epsilon_j$, to be specified later.

By conditions {\rm (c)} and {\rm (d)} in Claim \ref{cl:ohi}, we may apply Theorem \ref{th:Divisor} to the holomorphic submersion $\phi\colon W\to\c^q$, the closed submanifold $V\subset\b_n$, the polynomially convex compact set $r_j\overline\b_n\cup L_j$ (see Remark \ref{rem:pol}), and the number $\epsilon'>0$. This gives a holomorphic submersion $f_j\colon \b_n\to\c^q$ with the following properties.
\begin{itemize}
\item[\rm (I)] $f_j(z)=0\in\c^q$ for all $z\in V$ and $f_j$ agrees with $\phi$ to any given finite order everywhere on $V$.
\smallskip
\item[\rm (II)] $|f_j(z)-\phi(z)|<\epsilon'$ for all $z\in r_j\overline\b_n\cup L_j$.
\smallskip
\item[\rm (III)] If $\phi$ admits a continuous extension to $\b_n$ with $\phi(z)\neq 0$ for all $z\in \b_n\setminus V$, then $f_j^{-1}(0)=V$.
\end{itemize}

Recall that conditions {\rm (4$_j$)}, {\rm (5$_j$)}, and {\rm (6$_j$)} concerning $\epsilon_j$ and  $L_j$ have been already ensured. We claim that if $\epsilon'>0$ is chosen sufficiently small, then {\rm (1$_j$)}, {\rm (2$_j$)}, {\rm (3$_j$)}, and {\rm (7$_j$)} hold true for the holomorphic submersion $f_j$. Indeed, {\rm (1$_j$)} is implied by Claim \ref{cl:ohi} {\rm (a)} and property {\rm (II)}. Condition {\rm (2$_j$)} follows from {\rm (2$_{j-1}$)}, Claim \ref{cl:ohi} {\rm (a)}, and {\rm (I)}. Properties {\rm (3$_{j-1}$)}, Claim \ref{cl:f0}, Claim \ref{cl:ohi} {\rm (a)} and {\rm (d)}, and {\rm (III)} guarantee {\rm (3$_j$)}. (Recall that {\rm (2$_{j-1}$)} and {\rm (3$_{j-1}$)} are satisfied even when $j=1$.) Finally, in order to check condition {\rm (7$_j$)} choose a point $z\in L_j=\Lambda_V\cup\Lambda_0$. If $z\in \Lambda_V$, then the second inequality in \eqref{eq:L1}, Claim \ref{cl:ohi} {\rm (a)}, and {\rm (II)} imply that $|f_j(z)|<\lambda_j$ provided that $\epsilon'>0$ is chosen sufficiently small; recall that $\Lambda_V$ is compact. If on the contrary $z\in \Lambda_0$, then Claim \ref{cl:ohi} {\rm (b)} and property {\rm (II)} ensure that $|f_j(z)|>1/\lambda_j$ provided that $\epsilon'$ is small enough. This shows {\rm (7$_j$)}.

In order to finish the proof of the lemma note that, assuming that $\epsilon'>0$ is chosen sufficiently small, condition {\rm (2$_{j-1}$)}, the first inequality in \eqref{eq:L1}, Claim \ref{cl:ohi} {\rm (a)} and {\rm (b)}, property {\rm (II)}, and the continuity of $f_j$ ensure the existence of an open neighborhood $O_j$ of $L_j$ in $R_j\b_n\setminus r_j\overline \b_n$ such that 
\begin{equation}\label{eq:Oj}
	f_j^{-1}(0)\cap O_j\setminus V=\varnothing.
\end{equation}
On the other hand, the set $O_i$ in the statement of the lemma is contained in the ball $R_i\b_n\subset r_j\b_n$ for all $i\in\{0,\ldots,j-1\}$ (recall that $O_0=\varnothing$). Therefore, for a sufficiently small choice of $\epsilon'>0$, condition {\rm (8$_{j-1}$)} (which holds true even when $j=1$), Claim \ref{cl:ohi} {\rm (a)}, {\rm (I)}, and {\rm (II)} ensure that $f_j^{-1}(0)\cap O_i\setminus V=\varnothing$ for all $i\in\{0,\ldots,j-1\}$. This and \eqref{eq:Oj} give {\rm (8$_j$)}.

This closes the induction and thus concludes the proof of Lemma \ref{lem:ML}.
\end{proof}
%
%
With Lemma \ref{lem:ML} in hand, the proof of Theorem \ref{th:MR} is completed as follows.
Properties {\rm (1$_j$)} and {\rm (4$_j$)} imply that there is a limit map
\begin{equation}\label{eq:f}
	f=\lim_{j\to\infty}f_j\colon \b_n\to\c^q
\end{equation}
which is holomorphic and satisfies
\begin{equation}\label{eq:f-fj}
	|f(z)-f_{j-1}(z)|<2\epsilon_j<\epsilon_{j-1}\quad \text{for all }z\in r_j\overline\b_n,\; j\in\N.
\end{equation}
Since $K\subset r_1\b_n\cap U$ and $f_0|_U=h$ (see Claim \ref{cl:f0}) we infer from \eqref{eq:f-fj} that $|f(z)-h(z)|<\epsilon_0=\epsilon$ for all $z\in K$. This is condition {\it ii)} in the theorem.

From \eqref{eq:f-fj} and properties {\rm (2$_j$)}, {\rm (3$_j$)}, and {\rm (5$_j$)} we obtain that $f$ is a holomorphic submersion, $f(z)=0\in\c^q$ for all $z\in V$, $f$ agrees with $h$ to the given finite order everywhere on $V$, and, if $f_0^{-1}(0)=V$ (equivalently, if $h$ extends to a continuous map $\b_n\to\c^q$ which does not take the value $0\in\c^q$ off $V$; see Claim \ref{cl:f0}), then $f^{-1}(0)= V$. This shows {\it i)} and the final assertion in the statement of the theorem.

In order to check condition {\it iii)}, pick $\lambda>0$ and, without loss of generality, assume that $\lambda<1$. 
Since $\lim_{j\to\infty} \lambda_j=\lim_{j\to\infty} \epsilon_j=0$ (see \eqref{eq:llj0} and {\rm (4$_j$)}), there is  $j_\lambda\in\n$ so large that
\begin{equation}\label{eq:lambdaj<lambda}
	0<\lambda_j+\epsilon_j<\lambda<\frac1{\lambda}<\frac1{\lambda_j}-\epsilon_j
	\quad \text{for all $j\ge j_\lambda$.}
\end{equation}
We claim that $j_\lambda$ satisfies the first part of {\it iii)}. Indeed, suppose that for some $j\ge j_\lambda$ there is a path $\gamma\colon[0,1]\to \b_n$ such that $|\gamma(0)|\le r_j$, $|\gamma(1)|\ge R_j$, and
\begin{equation}\label{eq:lambda}
	\lambda\le|f(\gamma(t))|\le\frac1{\lambda}\quad \text{for all }t\in[0,1].
\end{equation}
Up to passing to a subpath we may assume that $\gamma([0,1]) \subset r_{j+1}\b_n$. This inclusion, \eqref{eq:f-fj}, \eqref{eq:lambdaj<lambda}, and \eqref{eq:lambda} imply that $\lambda_j< |f_j(\gamma(t))|< 1/\lambda_j$ for all $t\in[0,1]$, and hence $\gamma([0,1])\cap L_j=\varnothing$ by {\rm (7$_j$)}. Thus,  {\rm (6$_j$)} ensures that $\length(\gamma)> \delta_j$, which proves the first part of {\it iii)}. For the second part, choose a proper path $\beta\colon [0,1)\to\b_n$ and assume that $|f|$ is bounded above and bounded away from $0\in\r$  on $\beta([0,1))$. By this assumption, there is a number $\lambda$, $0<\lambda<1$, so small that
\begin{equation}\label{eq:beta}
	\lambda\le|f(\beta(t))|\le\frac1{\lambda}\quad \text{for all }t\in[0,1).
\end{equation}
 Let $j_\lambda$ be the integer furnished by the first part of {\it iii)} for this particular number $\lambda>0$ and let $j_0\ge j_\lambda$ be an integer for which $|\beta(0)|<r_{j_0}$; recall that $\beta(0)\in\b_n$ and $\lim_{j\to\infty} r_j=1$. Since the path $\beta\colon[0,1)\to\b_n$ is proper, we have that $\lim_{t\to1}|\beta(t)|=1$ and hence for each $j\ge j_0$ there is a closed interval $[a_j,b_j]\subset [0,1)$ such that $|\beta(a_j)|\le r_j$, $|\beta(b_j)|\ge R_j$, and the family of intervals $[a_j,b_j]$ $(j\ge j_0)$ are pairwise disjoint. Taking into account that the sequence of positive numbers $\delta_j$ $(j\in\n)$ is increasing, it follows from \eqref{eq:beta} and the first part of {\it iii)} that
\[
	\length(\beta)\ge\sum_{j\ge j_0}\length(\beta|_{[a_j,b_j]})\ge
	\sum_{j\ge j_0}\delta_j\ge \delta_{j_0}\sum_{j\ge 1} 1=+\infty.
\]
This completes the proof of condition {\it iii)}.

Finally note that, by properties {\rm (6$_j$)}, in order to check the first part of {\it iv)} it suffices to show that 
\begin{equation}\label{eq:VLj}
	L_i\cap f^{-1}(0)\setminus V=\varnothing\quad \text{for all $i\in\n$.}
\end{equation}
Reason by contradiction and assume that $f(z_0)=0\in\c^q$ for some $z_0\in L_i\setminus V$, $i\in\n$. Since $f\colon\b_n\to\c^q$ is a submersion, $f(z_0)=0$, and $O_i\setminus V$ is an open neighborhood of $z_0$ in $\b_n$, we infer that $f(O_i\setminus V)$ is a neighborhood of $0$ in $\c^q$. It follows that there is $r>0$ such that the ball $r\b_q$ is contained in $f(O_i\setminus V)\subset\c^q$. Therefore, since $f=\lim_{j\to\infty} f_j$ and every $f_j$ is submersive, there is an integer $m\ge i$ so large that $f_j(O_i\setminus V)\supset \frac{r}2\b_q\ni 0$ for all $j\ge m$, which is not possible in view of {\rm (8$_j$)}. This shows \eqref{eq:VLj} and therefore the first part of condition {\it iv)} in the theorem. The second part is seen in the same way that the second part of {\it iii)}.

This concludes the proof of Theorem \ref{th:MR}.
\end{proof}

\begin{proof}[Proof of Theorem \ref{th:intro-q}]
Let $V$ be as in Theorem \ref{th:intro-q}; i.e., a smooth closed complex submanifold of pure codimension $q$ in $\b_n$  $(1\le q<n)$ which is contained in a fibre of a holomorphic submersion $h\colon \b_n\to\c^q$. Assume without loss of generality that $V\subset h^{-1}(0)$. Thus, since both $V$ and $h^{-1}(0)$ are smooth closed complex submanifolds of pure codimension $q$ in $\b_n$, it follows that $V$ is open and closed in $h^{-1}(0)$ and hence a union of connected components of $h^{-1}(0)$.  It is then clear that we may apply Theorem \ref{th:MR} to $V$ and $h$ with $U=\b_n$ and 
$\theta=dh$; the rest of initial objects in the theorem are irrelevant to ensure the conclusion of Theorem \ref{th:intro-q} and shall not be specified. Let $f\colon \b_n\to\c^q$ be the holomorphic submersion obtained in this way.

We claim that $f$ satisfies the conclusion of Theorem \ref{th:intro-q}. Indeed, by properties {\it iii)} and {\it iv)} in Theorem \ref{th:MR}, we infer that the connected components of the fibres $f^{-1}(c)$ $(c\in\c^q)$ of $f$ are all, except perhaps those contained in $V\subset f^{-1}(0)$, complete; i.e.\ conditions {\rm (ii)} and {\rm (iii)} in Theorem \ref{th:intro-q}. On the other hand, Theorem \ref{th:MR} {\it i)} trivially implies Theorem \ref{th:intro-q} {\rm (i)}. If in addition $h$ admits a continuous extension to $\b_n$ with $h(z)\neq 0\in\c^q$ for all $z\in\b_n\setminus V$, then the final assertion in Theorem \ref{th:MR} enables us to choose $f$ with $f^{-1}(0)=V$, thereby proving the final assertion in Theorem \ref{th:intro-q}.

This completes the proof of Theorem \ref{th:intro-q}.
\end{proof}

Globevnik constructed in \cite{Globevnik2015AM} a holomorphic function on $\b_n$ whose real part, hence also whose norm, is unbounded on any proper path $[0,1)\to\b_n$ with finite length. 
Note that, since completeness of $V$ in Theorem \ref{th:MR} is not assumed, the function $f$ which we obtain does not enjoy this property.
Moreover, our method of proof does not seem to provide any information in this direction when the given submanifold is complete, and hence the question 
whether every smooth complete closed complex hypersurface in $\b_n$ is a level set of a holomorphic function on $\b_n$ which is unbounded in norm on every proper path $[0,1)\to\b_n$ with finite length
remains open.
Nevertheless, a slight modification of the proof of Theorem \ref{th:MR} enables us to prove the following extension of Corollary \ref{co:k=0}.
\begin{proposition}\label{pro:MR}
Let $0<r_1<R_1<r_2<R_2<\cdots$ be an increasing sequence with $\lim_{j\to\infty} r_j=1$. 
Then, for any increasing sequence $0<\delta_1<\delta_2<\cdots$ there is a holomorphic submersion $f\colon \b_n\to\c^q$ $(1\le q<n)$ satisfying the following condition: for any $\lambda>0$ there is  $j_\lambda\in\n$ for which if $\gamma\colon[0,1]\to \b_n$ is a path such that $|\gamma(0)|\le r_j$ and $|\gamma(1)|\ge R_j$ for some $j\ge j_\lambda$, and 
\[
	|f(\gamma(t))| < \frac1{\lambda} \quad \text{for all $t\in [0,1]$},
\]
then $\length (\gamma)>\delta_j$. 
In particular, $|f|$ is unbounded on any proper path $[0,1)\to\b_n$ with finite length, and hence the fibres $f^{-1}(c)$ $(c\in\c^q)$ of $f$ are all complete.
\end{proposition}
The novelty of Proposition \ref{pro:MR} with respect to Theorem \ref{th:MR} is that it also ensures that those proper paths $[0,1)\to \b_n$ on which $|f|$ is bounded above, but perhaps unbounded away from zero, have infinite length. We achieve this property at the cost of loosing the control on the fibre of $0\in\c^q$ under the submersion $f$. 
In the special case when $q=1$, Proposition \ref{pro:MR} was proved by Globevnik in \cite{Globevnik2016MA} but without ensuring that the function $f\in\Ocal(\b_n)$ be submersive (i.e., noncritical).
\begin{proof}[Sketch of proof]
We just point out the modifications to the proof of Theorem \ref{th:MR} which are needed to prove the proposition. On the one hand, conditions {\rm (2$_j$)}, {\rm (3$_j$)}, and {\rm (8$_j$)} in Lemma \ref{lem:ML} are not required (and do not make sense) in the current framework. On the other hand, we replace property {\rm (7$_j$)} in Lemma \ref{lem:ML} by the following more precise one:
\begin{itemize}
\item[\rm (7$_j'$)] If $z\in L_j$, then  $|f_j(z)|> 1/\lambda_j$.
\end{itemize}
This gives the conclusion of the proposition exactly in the same way that {\rm (7$_j$)} gives condition {\it iii)} in the proof of Theorem \ref{th:MR}; we omit the details.

It remains to explain how to ensure {\rm (7$_j'$)} in the inductive step. For that, using the notation in the proof of Lemma \ref{lem:ML}, it suffices to choose $\Lambda_V=\varnothing$ and $\Lambda_0=L_j$ (as it is natural since $V=\varnothing$ in our current framework), where $L_j$ is the tangent labyrinth in $R_j\b_n\setminus r_j\overline\b_n\subset\b_n$ provided by Lemma \ref{lem:labyrinth}, and follow that proof without any modification but without paying attention to the zero fibre of the submersions in the inductive construction. Finally, in order to prove the final assertion in the corollary we adapt, straightforwardly, the argument in the proof of the second part of condition {\it iii)} in Theorem \ref{th:MR}.
\end{proof}


\section{Proof of Theorem \ref{th:intro-q=1}}\label{sec:q=1}

We begin with some preparations. Let $X$ be a Stein manifold and $f\in\Ocal(X)$. A point $x\in X$ is said to be a {\em critical point} of $f$ if $df_x=0$; the set 
\[
	{\rm Crit}(f)=\{x\in X\colon df_x=0\}
\]
of all such points
is called the {\em critical locus} of the holomorphic function $f$. 
An inspection of the proofs of Theorem 2.1 in \cite{Forstneric2003AM} and Theorem 1.1 and Corollary 1.2 in \cite{Forstneric2018PAMS} gives the following extension of Theorem \ref{th:Divisor} in the case when $q=1$ and the second cohomology group with integer coefficients of $X$ vanishes.

\begin{theorem}[Forstneri\v c-Oka-Weil-Cartan theorem for functions]
\label{th:Divisor-1}
Let $X$ be a Stein manifold with $H^2(X;\z)=0$,
let $V$ be a closed complex hypersurface (possibly with singularities) in $X$, let $K\subset X$ be an $\Ocal(X)$-convex compact set, let $U\subset X$ be a neighborhood of $V\cup K$ in $X$, and assume that there is $h\in\Ocal(U)$ such that $V=h^{-1}(0)$ and ${\rm Crit}(h)\subset U$ is a closed discrete subset contained in $V\cup K$. Then, for any $\epsilon>0$ and $s\in\n$ there is $\wt h\in\Ocal(X)$ such that $V=\wt h^{-1}(0)$, $\wt h-h$ vanishes to order $s$ everywhere on $V$, $|\wt h(x)-h(x)|<\epsilon$ for all $x\in K$, and ${\rm Crit}(\wt h)={\rm Crit}(h)$.
\end{theorem}
The main novelty of Theorem \ref{th:Divisor-1} with respect to Theorem \ref{th:Divisor} (the Forstneri\v c-Oka-Weil-Cartan theorem for submersions) is, besides that the former also deals with functions with critical points, that there is no need to ensure the existence of a $1$-frame on $X$ which be compatible with $h$ on $U$. The fact that such always exists is guaranteed by the condition $H^2(X;\z)=0$.

Since the Euclidean ball $\b_n\subset\c^n$ $(n\ge 2)$ is contractible, we have that $H^2(\b_n;\z)=0$. Therefore, Theorem \ref{th:Divisor-1} applies with $X=\b_n$.

In this section we prove the following more precise version of Theorem \ref{th:intro-q=1}.

%
%
\begin{theorem}\label{th:MR-q=1}
Let $V$ be a closed complex hypersurface in $\b_n$ $(n\ge 2)$ and denote by $V_{\rm sing}$ its singular set. Let $P\subset\b_n$ be a closed discrete subset with $P\cap V=\varnothing$, let
 $U$ be an open neighborhood of $V$ in $\b_n$, and assume that there is $h\in\Ocal(U)$ with 
 \begin{equation}\label{eq:hhh}
 	\text{$h^{-1}(0)=V$\quad and\quad ${\rm Crit}(h)=V_{\rm sing}\cup (P\cap U)$.}
\end{equation}
Also choose a polynomially convex compact set $K\subset U$ and an increasing sequence $0<r_1<R_1<r_2<R_2<\cdots$ with  $K\subset r_1\b_n$ and $\lim_{j\to\infty} r_j=1$. 
Then, for any $\epsilon>0$ and any increasing sequence $0<\delta_1<\delta_2<\cdots$ there is $f\in\Ocal(\b_n)$ satisfying the following conditions.
\begin{itemize}
\item[\rm (i)] $f^{-1}(0)=V$ and $f-h$ vanishes to any given finite order everywhere on $V$. 
\smallskip
\item[\rm (ii)] $|f(z)-h(z)|<\epsilon$ for all $z\in K$.
\smallskip
\item[\rm (iii)] ${\rm Crit}(f)=V_{\rm sing}\cup P$. 
\smallskip
\item[\rm (iv)] For any $\lambda>0$ there is $j_\lambda\in\n$ for which if $\gamma\colon[0,1]\to \b_n$ is a path such that $|\gamma(0)|\le r_j$ and $|\gamma(1)|\ge R_j$ for some $j\ge j_\lambda$, and 
\[
	\lambda\le|f(\gamma(t))|\le \frac1{\lambda}\quad \text{for all $t\in [0,1]$},
\]
then $\length (\gamma)>\delta_j$. In particular, every proper path $[0,1)\to\b_n$ on which $|f|$ is bounded above and bounded away from zero has infinite length; and hence the level set $f^{-1}(c)\subset\b_n$ is complete for every $c\in f(\b_n)\setminus\{0\}$.
\end{itemize}
\end{theorem}
By the proofs of \cite[Theorem 2.1]{Forstneric2003AM} and \cite[Corollary 1.2]{Forstneric2018PAMS} (see Theorem \ref{th:Divisor-1}), for any $V$, $P$, and $U$ as in the statement of Theorem \ref{th:MR-q=1} there are holomorphic functions $h\colon U\to\c$ satisfying the requirements in \eqref{eq:hhh}. Therefore, the theorem applies to all closed complex hypersurfaces (possibly with singularities) in $\b_n$. 

It is clear in view of condition {\rm (iii)} that if the hypersurface $V$ in Theorem \ref{th:MR-q=1} is smooth and we choose $P=\varnothing$, then the holomorphic function $f\colon\b_n\to\c$ which we obtain is noncritical: ${\rm Crit}(f)=\varnothing$. In this case, the family of connected components of the level sets of $f$ is a nonsingular holomorphic foliation of $\b_n$ by smooth closed complex hypersurfaces all which, except perhaps those contained in $V$, are complete.

The proof of Theorem \ref{th:MR-q=1} follows very closely the one of Theorem \ref{th:MR} but using Theorem \ref{th:Divisor-1} instead of Theorem \ref{th:Divisor}.
Recall that a holomorphic function on a complex manifold is a submersion if and only if it has no critical points. As pointed out above, in this framework we do not need to deal with coframes on $\b_n$, and hence we may perform the same argument as that in the proof of Theorem \ref{th:MR} but without needing to arrange an analogue of condition {\rm (c)} in Claim \ref{cl:ohi}. Also the analogue of condition {\rm (d)} in that claim is now ensured from the fact that $H^2(\b_n;\z)=0$. 

%
%
\begin{proof}[Proof of Theorem \ref{th:MR-q=1}]
By Theorem \ref{th:Divisor-1} we may assume without loss of generality that $h\in\Ocal(\b_n)$, $h^{-1}(0)=V$, and ${\rm Crit}(h)=V_{\rm sing}\cup P$; see also \cite[Corollary 2.2]{Forstneric2003AM}.
Since $V_{\rm sing}\cup P$ is a closed discrete subset of $\b_n$, we may also assume that 
\begin{equation}\label{eq:P}
	(V_{\rm sing}\cup P)\cap R_j\overline\b_n\setminus r_j\b_n=\varnothing\quad \text{for all $j\in\n$;}
\end{equation} 
otherwise we just replace each pair $r_j<R_j$ by another one $r_j'<R_j'$ satisfying $r_j\le r_j'<R_j'\le R_j$ and the above condition.

Call $f_0=h$, $\epsilon_0=\epsilon$, and $L_0=\varnothing$. Choose a sequence of positive numbers $1>\lambda_1>\lambda_2>\cdots$ with $\lim_{j\to\infty}\lambda_j=0$.
Reasoning as in the proof of Lemma \ref{lem:ML}, we may inductively construct a sequence $S_j=\{f_j,\epsilon_j,L_j\}$ $(j\in\n)$, where
 $f_j\colon \b_n\to\c$ is a holomorphic function,
 $\epsilon_j>0$ is a number, and
 $L_j$ is a tangent labyrinth in $R_j\b_n\setminus r_j\overline \b_n$ (see Definition \ref{def:labyrinth}),
 such that the following hold for all $j\in\n$.
\begin{itemize}
\item[\rm (1$_j$)] $|f_j(z)-f_{j-1}(z)|<\epsilon_j$ for all $z\in r_j\overline\b_n$.
\smallskip
\item[\rm (2$_j$)] $f_j^{-1}(0)=V$ for all $z\in V$  and $f_j-h$ vanishes to the given finite order everywhere on $V$.
\smallskip
\item[\rm (3$_j$)] ${\rm Crit}(f_j)=V_{\rm sing}\cup P$. 
\smallskip
\item[\rm (4$_j$)] $\displaystyle 0<\epsilon_j<\epsilon_{j-1}/2$.
\smallskip
\item[\rm (5$_j$)] If $\gamma\colon [0,1]\to \b_n$ is a path such that $|\gamma(0)|\le r_j$, $|\gamma(1)|\ge R_j$, and $\gamma([0,1])\cap L_j=\varnothing$, then $\length(\gamma)> \delta_j$.
\smallskip
\item[\rm (6$_j$)] If $z\in L_j$, then either $|f_j(z)|<\lambda_j$ or $|f_j(z)|>1/\lambda_j$.
\end{itemize} 
(Comparing with the list of conditions in Lemma \ref{lem:ML}, we have merged {\rm (2$_j$)} and {\rm (3$_j$)} there into the single one {\rm (2$_j$)} here, whereas we have not included now any analogue of {\rm (8$_j$)}; these simplifications can be done since in the current framework we have $f_0^{-1}(0)=V$ and complete control on the zero level set of all the functions in the deformation procedure. Moreover, we have replaced condition {\rm (5$_j$)} in Lemma \ref{lem:ML} by the new {\rm (3$_j$)} here.)

Note that the triple $S_0=\{f_0,\epsilon_0,L_0\}$ satisfies {\rm (2$_0$)}, {\rm (3$_0$)}, and {\rm (6$_0$)}, whereas {\rm (1$_0$)}, {\rm (4$_0$)}, and {\rm (5$_0$)} are vacuous. The details of the inductive construction are very similar to those in the proof of Theorem \ref{th:MR} and we do not include them. We just point out that an analogue of Claim \ref{cl:ohi} (but only ensuring conditions {\rm (a)} and {\rm (b)}; see the discussion preceding this proof) can be arranged by defining
\[
	W\ni z\longmapsto \phi(z)=\left\{
	\begin{array}{ll}
	f_{j-1}(z) & \text{if }z\in \b_n\setminus\Delta_2 \medskip	
	\\
	\phi_0(z) & \text{if }z\in \mathring \Delta_1,
	\end{array}
	\right.
\]
where $W=(\b_n\setminus\Delta_2)\cup\mathring\Delta_1\subset\b_n$ is chosen as in that claim (see \eqref{eq:Wdef}) and $\phi_0\colon \Delta_1\to\c$ is any noncritical holomorphic function with $|\phi_0(z)|>1/\lambda_j$ for all $z\in\Delta_1$. 
This choice and \eqref{eq:P} guarantee that $\phi\colon W\to\c$ satisfies {\rm (a)} and {\rm (b)} in Claim \ref{cl:ohi} and also that $\phi^{-1}(0)=V$ and ${\rm Crit}(\phi)=V_{\rm sing}\cup P\subset \b_n\setminus (R_j\overline \b_n\setminus r_j \b_n) \subset \b_n\setminus\Delta_2$. It turns out that Theorem \ref{th:Divisor-1} applied to $\phi$ furnishes $f_j\in\Ocal(\b_n)$ satisfying conditions {\rm (1$_j$)}--{\rm (6$_j$)} above.

By {\rm (1$_j$)} and {\rm (4$_j$)}, there is a limit holomorphic function 
$
	f=\lim_{j\to\infty} f_j\colon\b_n\to\c
$
which, by {\rm (2$_j$)}, {\rm (3$_j$)}, and Hurwitz's theorem, satisfies conditions {\rm (i)} and {\rm (iii)} in the theorem. The rest of conditions are checked as in the proof of Theorem \ref{th:MR}.
\end{proof}
%
%
\begin{proof}[Proof of Theorem \ref{th:intro-q=1}]
Let $V\subset \b_n$ be a closed complex hypersurface (possibly with singularities) and denote by $V_{\rm sing}$ its singular set. By \cite[Corollary 1.2]{Forstneric2018PAMS}, there is $h\in\Ocal(\b_n)$ such that $h^{-1}(0)=V$ and ${\rm Crit}(h)=V_{\rm sing}$. Applying Theorem \ref{th:MR-q=1} to $h$ with $P=\varnothing$ we obtain $f\in\Ocal(\b_n)$ satisfying the conclusion of Theorem \ref{th:intro-q=1} (the rest of the initial objects in the statement of Theorem \ref{th:intro-q=1} are irrelevant to this aim).
\end{proof}

To finish this section it is perhaps worth stating the following result for future reference; it is trivially ensured by inspection of the proofs of Theorem \ref{th:MR}, Proposition \ref{pro:MR}, and Theorem \ref{th:MR-q=1}. 
%
%
\begin{remark}
Theorem \ref{th:MR}, Proposition \ref{pro:MR}, and Theorem \ref{th:MR-q=1} remain to hold true if one replaces $\b_n$ by $\c^n$ and takes the sequence $0<r_1<R_1<r_2<R_2<\cdots$ with 
$\lim_{j\to\infty} r_j=+\infty$. 
\end{remark}


\section{Complete complex hypersurfaces in a Stein manifold equipped with a Riemannian metric 
come in foliations}\label{sec:pseudo}

Once the existence of complete closed complex hypersurfaces in the ball of $\c^n$ $(n\ge 2)$ was known \cite{Globevnik2015AM}, it naturally appeared the question whether there is a general class of domains admitting this type of hypersurfaces. It is very clear that the arguments in \cite{Globevnik2015AM,AlarconGlobevnikLopez2019Crelle,AlarconGlobevnik2017C2} may be adapted to construct complete closed complex hypersurfaces in any (geometrically) convex domain of $\c^n$ (this is done in \cite{AlarconLopez2016JEMS} for the case when $n=2$ with a different technique). Going further in this direction, Globevnik \cite{Globevnik2016MA} proved that every pseudoconvex domain $\Omega\subset\c^n$ admits a holomorphic function $f\in\Ocal(\Omega)$ whose real part is unbounded above on every proper path $[0,1)\to\Omega$ with finite length; this ensures the existence of a (possibly singular) holomorphic foliation of $\Omega$ by complete closed complex hypersurfaces. (Recall that a domain $\Omega\subset\c^n$ $(n\ge 2)$ is said to be {\em pseudoconvex} if there is a strictly plurisubharmonic exhaustion function $\Omega\to \r$. This happens if and only if $\Omega$ is a domain of holomorphy and if and only if $\Omega$ is a Stein manifold; not every $n$-dimensional Stein manifold is biholomorphic to a pseudoconvex domain in $\C^n$, though. Convex domains in $\c^n$ are pseudoconvex. We refer to Range \cite{Range1986GTM} and H\"ormander \cite{Hormander1990NH} for background on the subject.) Taking into account Sard's theorem, this shows that every pseudoconvex domain in $\c^n$ admits a smooth complete closed complex hypersurface; see \cite[Corollary 1.2]{Globevnik2016MA}. Furthermore, Charpentier and  Kosi\'nski \cite{CharpentierKosinski2019} have builded labyrinths of compact sets in any given pseudoconvex Runge domain $\Omega\subset\C^n$ with analogous properties as those of the tangent labyrinths in balls constructed in \cite{AlarconGlobevnikLopez2019Crelle} and explained in Subsec.\ \ref{ss:l}; see Lemma \ref{lem:labyrinth} and Remark \ref{rem:pol} and compare with \cite[Lemma 2.4]{CharpentierKosinski2019}. Making use of these new labyrinths by Charpentier and  Kosi\'nski, one can easily adapt the arguments in Section \ref{sec:proof} to prove an extension of Theorem \ref{th:intro-q} in which the ball $\b_n$ is replaced by an arbitrary pseudoconvex Runge domain of $\C^n$; in particular, one obtains that every such domain admits nonsingular holomorphic submersion foliations by smooth complete closed complex submanifolds of any pure codimension $q\in\{1,\ldots,n-1\}$. Likewise, an analogue of Theorem \ref{th:intro-q=1} for pseudoconvex domains follows in this way.

In this section, we go a bit further in this direction and show the following analogue of Theorem \ref{th:intro-q=1} in which the ball is replaced by an arbitrary Stein manifold equipped with a Riemannian metric.
%
%
\begin{theorem}\label{th:pseudo-f=0}
Let $X$ be a Stein manifold of dimension $n\ge 2$ equipped with a Riemannian metric $\ggot$.
If $V$ is a smooth closed complex hypersurface in $X$ such that the normal bundle $N_{V/X}$ of $V$ in $X$ is trivial, 
then there is a holomorphic function $f\in \Ocal(X)$ with the following properties.
\begin{itemize}
\item[\rm (i)]  $f(z)=0$ for all $z\in V$. 
\smallskip
\item[\rm (ii)] Every proper path $[0,1)\to X$ on which $|f|$ is bounded above and bounded away from zero has infinite length (with respect to the metric $\ggot$). In particular, the level set $f^{-1}(c)\subset X$ is complete for every $c\in f(X)\setminus\{0\}$.
\smallskip
\item[\rm (iii)] $f^{-1}(0)\setminus V$ is either the empty set or a complete closed complex hypersurface (possibly with singularities) in $X$.
\smallskip
\item[\rm (iv)] $f$ has no critical points on $V$.
\end{itemize}
\end{theorem}
The metric $\ggot$ in the theorem need not be complete; the conclusion is trivial in that case.
Note that {\rm (i)} and {\rm (iv)} imply that $V$ is a union of components of $f^{-1}(0)$.
Note also that, by these conditions and the Docquier and Grauert tubular neighborhood theorem  \cite{DocquierGrauert1960MA}, the assumption that $N_{V/X}$ be trivial is necessary in the theorem. 
I wish to thank a colleague who pointed out to me that Theorem \ref{th:pseudo-f=0} holds true for any Stein manifold equipped with a Riemannian metric, instead of just for pseudoconvex domains, thereby giving rise to enlargement of the scope of the result.

The components of the level sets $f^{-1}(c)$ $(c\in \c)$ of the function $f$ in the theorem form a (possibly singular) holomorphic foliation of $X$ by connected closed complex hypersurfaces all which, except perhaps those contained in $V\subset f^{-1}(0)$, are complete (with respect to the metric induced by $\ggot$). If $V$ is complete, then all the leaves in the foliation are complete.  Taking into account Sard's theorem, this shows that every Stein manifold of dimension $n\ge 2$ equipped with an arbitrary Riemannian metric  admits a smooth complete closed complex hypersurface.
 
By \cite[Theorem I]{Forstneric2003AM}, every Stein manifold 
admits a holomorphic function without critical points. However, our method of proof does not seem to enable us to ensure that the function $f\in\Ocal(X)$ furnished by Theorem \ref{th:pseudo-f=0} be noncritical; this is a weak point of this result in comparison with Theorem \ref{th:intro-q=1}. In particular, the following remains an open problem.
\begin{conjecture}
Every Stein manifold $X$ of dimension $n\ge 2$ equipped with a Riemannian metric admits a nonsingular holomorphic foliation by smooth complete closed complex hypersurfaces. 
\end{conjecture}

In the proof of Theorem \ref{th:pseudo-f=0} we shall use the methods developed in the previous sections together with the ideas in \cite{Globevnik2016MA}, the classical Cartan extension theorem, and an Oka-Weil-Cartan type theorem for noncritical holomorphic functions from \cite{Forstneric2003AM}.
\begin{proof}[Proof of Theorem \ref{th:pseudo-f=0}]
Since $X$ is an $n$-dimensional Stein manifold, there is a proper holomorphic embedding $\psi\colon X\to \c^N$ for some integer $N>n$ (see \cite{Narasimhan1960AJM}). Call 
\begin{equation}\label{eq:SA}
	\text{$\Sigma=\psi( X)$\quad and\quad $A=\psi(V)$}. 
\end{equation}
By the assumptions on $V$ and $ X$, we have that $\Sigma$ is a smooth connected closed complex submanifold in $\c^N$ and $A$ is a smooth closed complex hypersurface in $\Sigma$ such that the normal bundle $N_{A/\Sigma}=T\Sigma|_A/TA$ of $A$ in $\Sigma$ is trivial; in particular, both $\Sigma$ and $A$ are Stein manifolds.
By \cite[Proof of Corollary 2.4]{Forstneric2003AM}, there is a noncritical holomorphic function $h\in\Ocal(\Sigma)$ which defines $A$ in an open neighborhood of $A$ in $\Sigma$; i.e., $A$ is the zero level set of the restriction of $h$ to a neighborhood of $A$ in $\Sigma$.
By Cartan's extension theorem, $h$ extends to a function $g_0\in\Ocal(\c^N)$ and, since this extension remains noncritical everywhere in $\Sigma$, \cite[Theorem 2.1]{Forstneric2003AM} enables us to assume that $g_0\colon\c^N\to\c$ is noncritical as well. 

Let $0<r_1<R_1<r_2<R_2<\cdots$ be a divergent sequence such that $A\cap r_1\b_N\neq\varnothing$. 
Since $\psi\colon X\to\c^N$ is a proper map, the set 
\[
	K_j=\psi^{-1}(R_j\overline\b_N\setminus r_j\b_N)\subset  X
\]
is compact, and hence $|d\psi|$ attains its maximum there. Let $0<\delta_1<\delta_2<\cdots$ be a divergent sequence such that
\begin{equation}\label{eq:djgrande}
	\delta_j\ge  \max\{|d\psi_x|\colon x\in K_j\}\quad \text{for all }j\in\n.
\end{equation}

Call $L_0=\varnothing$ and $O_0=\varnothing$, and choose a decreasing sequence of positive numbers $1>\lambda_1>\lambda_2>\cdots$ with $\lim_{j\to\infty}\lambda_j=0$. Also fix an integer $s\ge 2$ and a sequence $\{\epsilon_j\}_{j\in\z_+}$ with $0<\epsilon_j<\epsilon_{j-1}/2$ for all $j\in\n$.
%
%
\begin{lemma}\label{lem:pseudo}
There is a sequence $S_j=\{g_j,L_j,O_j\}$ $(j\in\n)$, where
\begin{itemize}
\item $g_j\in\Ocal(\c^N)$ is a noncritical holomorphic function,
\smallskip
\item $L_j$ is a tangent labyrinth in $R_j\b_N\setminus r_j\overline \b_N$ (Definition \ref{def:labyrinth}), and
\smallskip
\item $O_j$ is a neighborhood of $\Sigma\cap L_j$ in $\Sigma\cap R_j\b_N\setminus r_j\overline \b_N$,
\end{itemize}
 such that the following conditions hold true for all $j\in\n$.
\begin{itemize}
\item[\rm (1$_j$)] $|g_j(z)-g_{j-1}(z)|<\epsilon_j$ for all $z\in r_j\overline\b_N$.
\smallskip
\item[\rm (2$_j$)] $g_j(z)=0$ for all $z\in A$ and $g_j-h$ vanishes to order $s$ everywhere on $A$.
\smallskip
\item[\rm (3$_j$)] If $\gamma\colon [0,1]\to \c^N$ is a path such that $|\gamma(0)|\le r_j$, $|\gamma(1)|\ge R_j$, and $\gamma([0,1])\cap L_j=\varnothing$, then $\length(\gamma)> \delta_j$.
\smallskip
\item[\rm (4$_j$)] If $z\in L_j$, then either $|g_j(z)|<\lambda_j$ or $|g_j(z)|>1/\lambda_j$.
\smallskip
\item[\rm (5$_j$)] $g_j^{-1}(0)\cap O_i\setminus A=\varnothing$ for all $i\in\{0,\ldots,j\}$.
\end{itemize}
\end{lemma}
\begin{proof}
We argue as in the proof of Lemma \ref{lem:ML} but using  \cite[Theorem 2.1]{Forstneric2003AM}. 
The basis of the induction is given by $S_0=\{g_0,L_0,O_0\}$. The only difference in the inductive step appears when trying to ensure condition {\rm (5$_j$)}; note that in the current framework $O_j$ is not a neighborhood of $L_j$ in $\c^N$ but just of $\Sigma\cap L_j$ in $\Sigma$. We explain how to arrange this issue.

Assume that for some $j\in\n$ we already have tuples $S_i=\{g_i,L_i,O_i\}$ enjoying the corresponding properties for all $i\in\{0,\ldots,j-1\}$. Since $h\colon \Sigma\to\c$ is noncritical, the latter assertion in {\rm (2$_{j-1}$)} ensures that $d(g_{j-1}|_\Sigma)_z\neq 0$ for all $z\in A$, and hence, by continuity, also for all $z$ in an open neighborhood $W_j$ of $A$ in $\Sigma$. Therefore, we have that $A\subset (g_{j-1}|_\Sigma)^{-1}(0)$ (see {\rm (2$_{j-1}$)}) and both $A$ and $(g_{j-1}|_\Sigma)^{-1}(0)$ are smooth closed complex hypersurfaces in $\Sigma$. It turns out that $A$ is a union of components of $(g_{j-1}|_\Sigma)^{-1}(0)$, and thus  $g_{j-1}|_\Sigma$ defines $A$ on a neighborhood $U_j\subset W_j$ of $A$ in $\Sigma$:
\begin{equation}\label{eq:A}
	A=(g_{j-1}|_\Sigma)^{-1}(0)\cap U_j.
\end{equation}

We now apply Lemma \ref{lem:labyrinth} to obtain a tangent labyrinth $L_j$ in $R_j\b_N\setminus r_j\overline \b_N$ satisfying {\rm (3$_j$)} and the following condition: if $T$ is a component of $L_j$ which has nonempty intersection with $A$, then $\Sigma\cap T\subset U_j$, and hence 
\begin{equation}\label{eq:A1}
	g_{j-1}(z)\neq 0\quad \text{for all } z\in \Sigma\cap T\setminus A.
\end{equation} 
By \eqref{eq:A} and the compactness of $\Sigma\cap\overline  R_j\b_N\setminus r_j\b_N$, to ensure this condition it suffices to choose $L_j$ with all of its connected components having small enough diameter. Next, reasoning as in the proof of Lemma \ref{lem:ML} (hence, again, choosing $L_j$ such that all of its components intersecting $A$ have small enough diameter), \cite[Theorem 2.1]{Forstneric2003AM} provides us with a noncritical function $g_j\in\Ocal(\c^N)$ satisfying conditions {\rm (1$_j$)}, {\rm (2$_j$)}, and {\rm (4$_j$)}, and such that $g_j(z)\neq 0$ for all $z\in\Sigma\cap L_j\setminus A$ (take into account \eqref{eq:A1} and cf.\ \eqref{eq:L1} and conditions {\rm (a)} and {\rm (b)} in  Claim \ref{cl:ohi}). It follows that there is an open neighborhood $O_j$ of $\Sigma\cap L_j$ in $\Sigma\cap R_j\b_N\setminus r_j\overline \b_N$ such that
\begin{equation}\label{eq:OiSigma}
	g_j^{-1}(0)\cap O_j\setminus A=\varnothing.
\end{equation}
If the approximation of $g_{j-1}$ by $g_j$ on $r_j\overline\b_N\supset\bigcup_{i=0}^{j-1} O_i$ is close enough (see {\rm (1$_j$)}), then conditions \eqref{eq:OiSigma} and {\rm (5$_{j-1}$)} ensure {\rm (5$_j$)}. This completes the proof.
\end{proof}

We continue the proof of Theorem \ref{th:pseudo-f=0}. Since $0<\epsilon_j<\epsilon_{j-1}/2$ for all $j\in\n$, it follows from {\rm (1$_j$)} that there is a limit holomorphic function
$
	g=\lim_{j\to\infty} g_j\colon\c^N\to\c.
$ 
 Moreover, taking into account that each $g_j\in\Ocal(\c^N)$ is noncritical and conditions {\rm (5$_j$)}, Hurwitz's theorem guarantees that $g$ is noncritical and 
\[
	g(z)\neq 0\quad \text{for all $z\in \big(\Sigma\cap \bigcup_{j\ge 1} L_j\big)\setminus A$}.
\]
From this condition and properties {\rm (1$_j$)}--{\rm (5$_j$)}, we infer the following.
\begin{itemize}
\item[\rm (a)] $g(z)=0$ for all $z\in A$ and $g-h$ vanishes to order $s\ge 2$ everywhere on $A$. (Recall that $A\subset \Sigma$ and  $h=g_0|_\Sigma$.)
\smallskip
\item[\rm (b)] For any $\lambda>0$ there is $j_\lambda\in\n$ for which if $\gamma\colon[0,1]\to \c^N$ is a path such that $|\gamma(0)|\le r_j$ and $|\gamma(1)|\ge R_j$ for some $j\ge j_\lambda$, and 
$\lambda\le|g(\gamma(t))|\le \frac1{\lambda}$ for all $t\in [0,1]$,
then $\length(\gamma)>\delta_j$.
\smallskip
\item[\rm (c)] If $\gamma\colon[0,1]\to g^{-1}(0)\cap\Sigma\setminus A$ is a path such that $|\gamma(0)|\le r_j$ and $|\gamma(1)|\ge R_j$ for some $j\ge 1$, then $\length (\gamma)>\delta_j$.
\end{itemize}
This is checked as in the proof of Theorem \ref{th:MR} and we omit the details. 

Consider the holomorphic function
\begin{equation}\label{eq:fgpsi}
	f=g\circ \psi\colon X\to\c,
\end{equation}
where $\psi\colon  X\to \c^N$  is the proper holomorphic embedding that was chosen at the very beginning of the proof. We claim that $f$ satisfies the conclusion of the theorem.

Indeed, to check {\rm (ii)} it suffices to show that if $\gamma\colon[0,1]\to X$ is a path for which there is $\lambda>0$ such that $|\psi(\gamma(0))|\le r_j$ and $|\psi(\gamma(1))|\ge R_j$ for some $j\ge j_\lambda$, and
$\lambda\le |f(\gamma(t))|\le\frac1{\lambda}$ for all $t\in [0,1]$,
then $\length_\ggot(\gamma)>1$. (Here $j_\lambda$ is the integer provided by condition {\rm (b)}  and $\length_\ggot$ denotes the length operator in the Riemannian manifold $(X,\ggot)$.) So, let $\gamma$ be such a path and set $\wt \gamma=\psi\circ\gamma\colon [0,1]\to\c^N$. Since $g\circ\wt \gamma=g\circ \psi\circ\gamma=f\circ\gamma$, we have that $|\wt\gamma(0)|\le r_j$, $|\wt\gamma(1)|\ge R_j$, and
$\lambda\le |g(\wt\gamma(t))|\le\frac1{\lambda}$ for all $t\in [0,1]$,
and hence condition {\rm (b)} ensures that $\length(\wt\gamma)>\delta_j$. From this,  \eqref{eq:djgrande}, and the definition of $\wt \gamma$ we obtain that
\[
	\length_\ggot(\gamma)  \ge  \frac{\length(\wt \gamma)}{\max\{|d\psi_x|\colon x\in K_j\}}>1.
\]
This proves {\rm (ii)}; we refer to \cite[Proof of Lemma 2.3]{Globevnik2016MA} for further details in the case when $X$ is a pseudoconvex domain in $\C^n$. 

On the other hand, by \eqref{eq:SA}, the definition of $f$, and the bijectivity of $\psi\colon X\to\Sigma$, it turns out that 
\[
	\psi(f^{-1}(0)\setminus V)=\psi(f^{-1}(0))\setminus \psi(V)=\Sigma\cap g^{-1}(0)\setminus A. 
\]
Taking into account this, a similar argument to the above one but using condition {\rm (c)} shows that every proper path $\gamma\colon[0,1)\to  X$ with $\gamma([0,1))\subset f^{-1}(0)\setminus V$ has infinite length with respect to the Riemannian metric $\ggot$.
 This ensures condition {\rm (iii)} in the statement of the theorem. 
 
Finally, conditions {\rm (i)} and {\rm (iv)} follow from {\rm (a)}, the definition of $f$ in \eqref{eq:fgpsi}, and the facts that $h\in\Ocal(\Sigma)$ is noncritical, $A=\psi(V)\subset\Sigma$, and $\psi\colon  X\to\Sigma$ is a regular biholomorphic map.
This concludes the proof of Theorem \ref{th:pseudo-f=0}.
\end{proof}


\subsection*{Acknowledgements}
The author was partially supported by the State Research Agency (SRA) and European Regional Development Fund (ERDF) via the grant no.\ MTM2017-89677-P, MICINN, Spain.

I wish to thank Franc Forstneri\v c and Josip Globevnik for their encouragement and for helpful remarks which led to improve the exposition, to the former also for explanations about the methods in the papers \cite{Forstneric2003AM,Forstneric2018PAMS}.




\noindent Antonio Alarc\'{o}n

\noindent Departamento de Geometr\'{\i}a y Topolog\'{\i}a e Instituto de Matem\'aticas (IEMath-GR), Universidad de Granada, Campus de Fuentenueva s/n, E--18071 Granada, Spain.

\noindent  e-mail: {\tt alarcon@ugr.es}

\end{document}